\documentclass[11pt]{amsart}
\usepackage{amsmath,amssymb,latexsym,soul,cite}
\usepackage{xcolor,enumitem,graphicx}

\usepackage{enumitem}

\definecolor{CherryRed}{rgb}{0.7,0.1,0.1}
\definecolor{EgyptBlue}{rgb}{0.1,0.1,0.6}

\usepackage[colorlinks=true,urlcolor=olive,
citecolor=CherryRed,linkcolor=EgyptBlue,linktocpage,pdfpagelabels,
bookmarksnumbered,bookmarksopen]{hyperref}
\usepackage[english]{babel}
\usepackage[T1]{fontenc}
\usepackage[hyperpageref]{backref}

\usepackage{epstopdf}
\usepackage{a4wide}
\newtheorem{thm}{Theorem}[section]

\theoremstyle{plain}
\newtheorem{lem}[thm]{Lemma}
\theoremstyle{plain}
\newtheorem{prop}[thm]{Proposition}
\theoremstyle{plain}

\theoremstyle{plain}
\theoremstyle{definition}
\newtheorem{definition}[thm]{Definition}
\newtheorem{rem}[thm]{Remark}

\newcommand{\R}{{\mathbb R}}
\newcommand{\capp}{{\mathrm{cap}_p}}

\usepackage{marginnote}

\title[On the higher Cheeger problem]{On the higher Cheeger problem}

\author[V. Bobkov]{Vladimir Bobkov}
\author[E. Parini]{Enea Parini}

\address[E. Parini]{Aix-Marseille Univ, CNRS, Centrale Marseille, I2M, 39 Rue Frederic Joliot Curie, 13453 Marseille, France}\email{enea.parini@univ-amu.fr}
\address[V. Bobkov]{Department of Mathematics and NTIS, Faculty of Applied Sciences, University of West Bohemia, Univerzitn\'i 8, 306 14 Plze\v{n}, Czech Republic} \email{bobkov@kma.zcu.cz}


\begin{document}

\begin{abstract}
	We develop the notion of higher Cheeger constants for a measurable set $\Omega \subset \mathbb{R}^N$. By the $k$-th Cheeger constant we mean the value \[h_k(\Omega) = \inf \max \{h_1(E_1), \dots, h_1(E_k)\},\] where the infimum is taken over all $k$-tuples of mutually disjoint subsets of $\Omega$, and $h_1(E_i)$ is the classical Cheeger constant of $E_i$. 
	We prove the existence of minimizers satisfying additional ``adjustment'' conditions and study their properties. 
	A relation between $h_k(\Omega)$ and spectral minimal $k$-partitions of $\Omega$ associated with the first eigenvalues of the $p$-Laplacian under homogeneous Dirichlet boundary conditions is stated. The results are applied to determine the second Cheeger constant of some planar domains.
\end{abstract}

\maketitle	

\section{Introduction}\label{sec:intro}

Let $\Omega \subset \mathbb{R}^N$ ($N \geq 2$) be a measurable set with positive Lebesgue measure, i.e., $|\Omega| > 0$. 
The \emph{Cheeger problem} consists in determining the value of the \emph{Cheeger constant} of $\Omega$, defined as 
$$
h_1(\Omega) := \inf \left\{ \frac{P(E)}{|E|}:~ E \subset \Omega,~ |E| > 0 \right\},
$$
and the associated minimizing sets, each of which is called \emph{Cheeger set}. Here $P(E) := P(E; \mathbb{R}^N)$ is the distributional perimeter of $E$ with respect to $\mathbb{R}^N$ (cf.\ \cite{giusti}). For an introduction to the Cheeger problem we refer to the expository articles \cite{Parini2011} and \cite{leonardi}, and the references therein.

A natural generalization of the above concept is the \emph{higher Cheeger problem}, where minimization takes place among $k$-tuples of mutually disjoint subsets of $\Omega$. More precisely, the \textit{$k$-th Cheeger constant} of $\Omega$ is defined as
\begin{equation} \label{higherconstant1}
h_k(\Omega) := \inf \left\{ \max_{i=1,\dots,k} \frac{P(E_i)}{|E_i|}:~ E_i \subset \Omega,~ |E_i| > 0 ~ \forall i,~ E_i \cap E_j = \emptyset ~ \forall i \neq j \right\},
\end{equation}
or, in an equivalent way, as
\begin{equation} \label{cheegerpartition}
h_k(\Omega) := \inf \left\{ \max_{i=1,\dots,k} h_1(E_i):~ E_i \subset \Omega,~ |E_i| > 0 ~ \forall i,~ E_i \cap E_j = \emptyset ~ \forall i \neq j \right\}.
\end{equation}
Besides having a geometric interest on its own, the Cheeger problem, as well as its generalization, arises in the study of the asymptotic behaviour of the eigenvalue problem for the $p$-Laplacian as $p$ tends to $1$. Assume that $\Omega$ is a bounded domain. We say that $\lambda \in \R$ is an \emph{eigenvalue} of the $p$-Laplacian if there exists a nontrivial weak solution $u \in W^{1,p}_0(\Omega)$, which is called \emph{eigenfunction}, to the problem
\begin{equation}\label{p-Laplacian}
\left\{
\begin{aligned}
-\Delta_p u &= \lambda |u|^{p-2} u &&\text{in } \Omega,\\
u &= 0 &&\text{on } \partial \Omega.
\end{aligned}
\right.
\end{equation}
Here $p \in (1,+\infty)$, and $\Delta_p u:=\text{div}(|\nabla u|^{p-2}\nabla u)$. The existence of sequences of eigenvalues of the $p$-Laplacian can be proven by means of minimax principles, see, e.g., \cite{garciaazoreroperal, drabrob1999}. In this paper, we will focus on the sequence $\{\lambda_k(p;\Omega)\}_{k \in \mathbb{N}}$ defined using the Krasnoselskii genus \cite[\S~5]{garciaazoreroperal}, which satisfies
$$ 
0 < \lambda_1(p;\Omega) < \lambda_2(p;\Omega) \leq \dots \leq \lambda_k(p;\Omega) \to +\infty \quad \text{as } k \to +\infty.
$$
As for the behaviour of \eqref{p-Laplacian} for $p \to 1$, it was proven in \cite[Corollary~6 and Remark~7]{kawohlfridman} that
\[ \lim_{p \to 1} \lambda_1(p;\Omega) = h_1(\Omega).\]
A few years later, the result was generalized to higher eigenvalues. In \cite[Theorem~5.5]{Parini} it was proven that
\begin{equation}\label{secondpto1} 
\lim_{p \to 1} \lambda_2(p;\Omega) = h_2(\Omega)
\end{equation}
and, more in general,
\begin{equation}\label{higherpto1} 
\limsup_{p \to 1} \lambda_k(p;\Omega) \leq h_k(\Omega) \quad \text{for } k \geq 3.
\end{equation}
We also mention the paper \cite{littigschur}, where the authors showed that 
\begin{equation*}\label{eq:limlambdak}
\lim\limits_{p \to 1} \lambda_k(p; \Omega) = \Lambda_k(\Omega)
\quad 
\text{for } k \in \mathbb{N},
\end{equation*}
where $\Lambda_k(\Omega)$ is the $k$-th Krasnoselskii eigenvalue of the $1$-Laplace operator (cf.\ \cite{chang2009}). 

The reason of the discrepancy between \eqref{secondpto1} and \eqref{higherpto1} is the fact that, while every eigenfunction $u_2$ associated to $\lambda_2(p;\Omega)$ has exactly two nodal domains (namely, connected components of the set $\{u_2 \neq 0\}$), an eigenfunction $u_k$ associated to $\lambda_k(p;\Omega)$ does not have, in general, $k$ nodal domains. Therefore, and also in view of \eqref{cheegerpartition}, it makes sense to introduce the spectral minimal $k$-partition problem for the $p$-Laplacian as
$$
\mathfrak{L}_k(p;\Omega) := \inf \left\{ \max_{i=1,\dots,k} \lambda_1(p; E_i):~ E_i \subset \Omega,~ |E_i| > 0 ~ \forall i,~ E_i \cap E_j = \emptyset ~ \forall i \neq j \right\},
$$
see Section \ref{sec:spectral_minimal_partitions} below for precise definitions. The existence, regularity and qualitative properties of $\mathfrak{L}_k(p; \Omega)$ in the case $p=2$ have been studied intensively nowadays, see, e.g., \cite{bucurbutazzohenrot,conti2005,helffer2006,helffer2010}.

One of the main results of this paper is Theorem \ref{thm:1}, where we prove that $k$-th Cheeger constant $h_k(\Omega)$ can be characterized as 
$$
\lim_{p \to 1} \mathfrak{L}_k(p;\Omega) = h_k(\Omega).
$$
Let us mention that a related spectral partitioning problem was studied in \cite{caroccia}. In that paper, the author investigated the limit as $p \to 1$ of the quantity
$$
\Lambda_k^{(p)}(\Omega) := \inf \left\{ \sum_{i=1}^{k} \lambda_1(p;E_i):~ E_i \subset \Omega,~ |E_i| > 0 ~ \forall i,~ E_i \cap E_j = \emptyset ~ \forall i \neq j \right\}
$$
and proved its convergence towards
\begin{equation}\label{def:H_k}
H_k(\Omega) := \inf \left\{ \sum_{i=1}^{k} h_1(E_i):~ E_i \subset \Omega,~ |E_i| > 0 ~ \forall i,~ E_i \cap E_j = \emptyset ~ \forall i \neq j \right\}.
\end{equation}
Existence and qualitative properties of minimizing $k$-tuples of sets for $H_k(\Omega)$ were also comprehensively studied in \cite{caroccia}.
On the other hand, the functional $\Lambda_k^{(p)}(\Omega)$ in the case $p=2$ also attracted the attention of a significant number of researchers, see, for instance,  \cite{bucurbutazzohenrot,conti2005,caffarelli,bourdin}. 

Expressions \eqref{higherconstant1} and \eqref{cheegerpartition} can be rewritten in shorter forms as
\begin{equation}\label{hkmax2}
h_k(\Omega) = \inf_{(E_1,\dots,E_k) \in \mathcal{E}_k} \max_{i=1,\dots,k} \frac{P(E_i)}{|E_i|}
~\quad \text{and} \quad~ 
h_k(\Omega) = \inf_{(E_1,\dots,E_k) \in \mathcal{E}_k} \max_{i=1,\dots,k} h_1(E_i),
\end{equation}
where $\mathcal{E}_k$ is the family of all $k$-tuples of mutually disjoint measurable subsets of $\Omega$ with positive Lebesgue measure. 
Note that each $E_i$ can be assumed to be connected.
A natural question is whether the infimum in \eqref{hkmax2} is actually attained. A first result about the existence of a minimizing $k$-tuple $(E_1, \dots, E_k) \in \mathcal{E}_k$ for $h_k(\Omega)$ is proved in \cite[Theorem~3.1]{Parini}. In what follows, a minimizer of $h_k(\Omega)$ will be called a $k$-tuple of multiple Cheeger sets, or, simply, a \textit{Cheeger $k$-tuple} of $\Omega$.
If $k=2$, we also call minimizers $E_1$ and $E_2$ \textit{coupled Cheeger sets}. In general, a $k$-tuple of multiple Cheeger sets for $h_k(\Omega)$ need not be unique, as the following example shows\footnote{The example, as well as Figure \ref{fig:Fig1}, has only an illustrative purpose, since the various assertions are not rigorously proven.}. Set $k=3$ and let $\Omega$ be a union of a square $K$, a disc $B$, and a negligibly thin channel $T$ which connects $K$ and $B$, see Fig.~\ref{fig:Fig1}. 
On the one hand, we can take a radius of $B$ sufficiently small to get $h_2(K) < h_1(B)$. On the other hand, we can take a radius of $B$ sufficiently large to get $h_1(B) < h_3(K)$. Thus, $h_3(\Omega) = h_1(B)$ and we have some freedom to vary sets $E_2, E_3 \subset K$ of a minimizer $(E_1, E_2, E_3) \in \mathcal{E}_3$ for $h_3(\Omega)$ without loss of minimizing property. 

\begin{figure}[ht]
	\centering
	\includegraphics[width=0.9\linewidth]{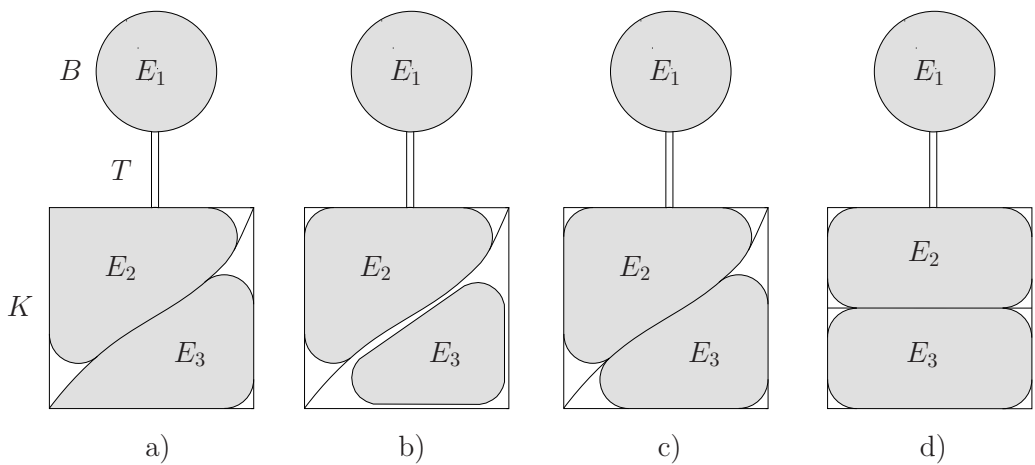}
	\caption{Assume $h_1(E_2), h_1(E_3) < h_1(E_1) = h_3(\Omega)$.
		a) A minimizer is not calibrable; 
		b) a minimizer is calibrable, but not a $1$-adjusted Cheeger triple; 
		c) a minimizer is a $1$- but not a $2$-adjusted Cheeger triple (the configuration in the square is not optimal); 
		d) a minimizer is a $2$-adjusted Cheeger triple.}
	\label{fig:Fig1}
\end{figure}

The above example also indicates that we cannot expect from an arbitrary taken Cheeger $k$-tuple of $\Omega$ to be regular enough. With this respect, it is reasonable to look for minimizers of $h_k(\Omega)$ satisfying some additional adjustment conditions. 
A first guess could be to require all sets of the Cheeger $k$-tuple to be \emph{calibrable} (cf. \cite{alter}), namely, each $E_i$ is a Cheeger set of itself:
$$
h_1(E_i) = \frac{P(E_i)}{|E_i|}
\quad \text{for all } i \in \{1,\dots,k\}.
$$
However, this is not completely satisfactory, as can be seen in the example in Fig.~\ref{fig:Fig1}. A more suitable requirement is defined as follows. We say that a minimizer $(E_1, \dots, E_k) \in \mathcal{E}_k$ of $h_k(\Omega)$ is a \textit{$1$-adjusted} Cheeger $k$-tuple if 
$$
h_1\biggl(\Omega \setminus \bigcup_{j \neq i} {E_j}\biggr) =
h_1(E_i) = 
\frac{P(E_i)}{|E_i|}
\quad \text{for all }
i \in \{1,\dots,k\},
$$
that is, each $E_i$ is a Cheeger set of $\Omega \setminus \bigcup_{j \neq i} {E_j}$. 

Analogously, we say that a minimizer $(E_1, \dots, E_k) \in \mathcal{E}_k$ of $h_k(\Omega)$ is a \textit{$2$-adjusted} Cheeger $k$-tuple if it is $1$-adjusted and, additionally,
\begin{align*}
h_2\biggl(\Omega \setminus \bigcup_{j \neq i_1, i_2} {E_j}\biggr) 
= h_2(E_{i_1} \cup E_{i_2}) 
&= \max\left\{\frac{P(E_{i_1})}{|E_{i_1}|}, \frac{P(E_{i_2})}{|E_{i_2}|}\right\} \\
&\text{for all mutually different } i_1, i_2 \in \{1,\dots,k\},
\end{align*}
that is, each pair $(E_{i_1}, E_{i_2})$ is a $1$-adjusted Cheeger couple of $\Omega \setminus \bigcup_{j \neq i_1, i_2} {E_j}$.

Proceeding in this way, we say that a minimizer $(E_1, \dots, E_k) \in \mathcal{E}_k$ of $h_k(\Omega)$ is a \textit{$n$-adjusted} Cheeger $k$-tuple for $n \in \{2, \dots, k\}$ if it is $(n-1)$-adjusted and, additionally,
\begin{align*}
h_n\biggl(\Omega \setminus \bigcup_{j \neq i_1, \dots, i_n} {E_j}\biggr)
=h_n(E_{i_1} \cup \cdots \cup E_{i_n})
&= \max\left\{\frac{P(E_{i_1})}{|E_{i_1}|}, \dots, \frac{P(E_{i_n})}{|E_{i_n}|}\right\} \\
&\text{for all mutually different } i_1, \dots, i_n \in \{1,\dots,k\},
\end{align*}
that is, each $n$-tuple $(E_{i_1}, \dots, E_{i_n})$ is a $(n-1)$-adjusted Cheeger $n$-tuple of $\Omega \setminus \bigcup_{j \neq i_1, \dots, i_n} {E_j}$. 

Equivalently, a minimizer $(E_1, \dots, E_k) \in \mathcal{E}_k$ of $h_k(\Omega)$ is a $n$-adjusted Cheeger $k$-tuple for $n \in \{1, \dots, k\}$ if any $l$-tuple $(E_{i_1}, \dots, E_{i_l}) \subset (E_1, \dots, E_k)$ with arbitrary $l \in \{1,\dots,n\}$ is a minimizer of $h_l\left(\Omega \setminus \bigcup_{j \neq i_1, \dots, i_l} {E_j}\right)$. 
Moreover, it is easy to see from the definitions that any $(k-1)$-adjusted Cheeger $k$-tuple is, in fact, $k$-adjusted. 

\medskip
Let us remark that these adjustment conditions are not necessary for the problem defined in \eqref{def:H_k}, since in this case the Cheeger constant of \textit{each} component $E_i$ of a minimizer $(E_1,\dots,E_k)$ contributes to the value of $H_k(\Omega)$, while in our problem $h_k(\Omega)$ is defined via the maximal Cheeger constant only. 

As for the existence of adjusted Cheeger $k$-tuples, our main result is Theorem \ref{thm:main1}, where we prove that, for any $k \in \mathbb{N}$, and for any $n \in \{1,\dots,k\}$, there exists a $n$-adjusted Cheeger $k$-tuple.

The advantage of working with adjusted Cheeger $k$-tuples is the fact that they satisfy the usual regularity properties of perimeter-minimizing sets, which are stated in detail in Theorem \ref{thm:interiorregularity}. In particular, each set of any $1$-adjusted Cheeger $k$-tuple can be approximated from the inside by a sequence of smooth sets in a suitable sense (see Proposition \ref{prop:approximation}). This result plays a crucial role for the characterization of $h_k(\Omega)$ as a limit of spectral minimal partitions, see Theorem \ref{thm:1}.

Qualitative properties of Cheeger $k$-tuples were also investigated in \cite{Parini}, although with a few imprecisions. In particular, \cite[Theorem 3.9]{Parini} is not correct, since the proof uses the erroneous fact that the perimeter of both Cheeger sets increases under volume-preserving deformations, which is in general not true. We correct that statement by our Proposition \ref{prop1}. 

\medskip
The paper is structured as follows. After proving the existence of adjusted Cheeger $k$-tuples for $h_k(\Omega)$ by means of inductive arguments, in Section \ref{sec:regularity} we investigate their regularity properties, and in Section \ref{sec:properties} we study further qualitative properties. In Section \ref{sec:spectral_minimal_partitions} we consider the spectral minimal partition problem for the $p$-Laplacian and study its limit as $p \to 1$. The last section is devoted to the computation of $h_2(\Omega)$ when $\Omega$ is a disc or an annulus.

\section[Existence of adjusted Cheeger k-tuples]{Existence of adjusted Cheeger $k$-tuples}\label{sec:existence}

\begin{thm}\label{thm:main1}
	Let $\Omega \subset \mathbb{R}^N$ be a bounded measurable set with positive Lebesgue measure.
	Then, for any $k \in \mathbb{N}$ and $n \in \{1,\dots,k\}$, there exists a $n$-adjusted Cheeger $k$-tuple.
\end{thm}
\begin{proof}
Let us denote, for simplicity,
$$
J(E) = \frac{P(E)}{|E|}.
$$
The proof is combinatorial and based on inductive arguments. 
Let us outline its scheme:
\begin{enumerate}[label={\rm(\Roman*)}]
	\item\label{thm:main1:I} 
	Show the existence of a $1$-adjusted Cheeger $k$-tuple for any $k \geq 1$.
	This is the base of induction with respect to the adjustment order. 
	\item\label{thm:main1:II} 
	Fix $n \geq 2$ and suppose that for any $n' < n$ there exists a $n'$-adjusted Cheeger $k$-tuple for any $k \geq n'$. 
	This is the induction hypothesis. 
	Then, prove the existence of a $n$-adjusted Cheeger $k$-tuple for any $k \geq n$. 
	This is the inductive step.
\end{enumerate}

In the steps \ref{thm:main1:I} and \ref{thm:main1:II} we again apply the induction as follows:
\begin{enumerate}[label={\rm(\arabic*)}]
	\item\label{thm:main1:1}  Fix $n \geq 1$ and show the existence of a $n$-adjusted Cheeger $(n+1)$-tuple. (Note that the existence of a $n$-adjusted Cheeger $n$-tuple is obvious: if $n=1$, then it is trivial; if $n \geq 2$, then it simply follows from the existence of a $(n-1)$-adjusted Cheeger $n$-tuple, which we know by the induction hypothesis of \ref{thm:main1:II}).
	This is the base of induction with respect to the index of the higher Cheeger constant.
	\item\label{thm:main1:2}  Fix $k \geq n+1$ and suppose that for any $k' \in \{n,\dots,k-1\}$ there exists a $n$-adjusted Cheeger $k'$-tuple.
	This is the induction hypothesis.
	Then, prove the existence of a $n$-adjusted Cheeger $k$-tuple. 
	This is the inductive step.
\end{enumerate}

Let us emphasize that the induction hypotheses are assumed to be satisfied for arbitrary bounded measurable $\Omega$ with positive Lebesgue measure.

\medskip
We now turn to details.

\ref{thm:main1:I} First we consider the case of $1$-adjusted Cheeger $k$-tuples.

\ref{thm:main1:1} Let $k=2$, and let $(C_1,C_2)$ be a minimizer of $h_2(\Omega)$. Without loss of generality, we can assume that $J(C_1)=h_2(\Omega)$. 
Observe that this does not mean, in general, that $h_1(\Omega \setminus C_2) = h_1(C_1) = J(C_1)$.
We will distinguish two cases:
\begin{enumerate}[label={\rm(\alph*)}]
	\item\label{thm:main1:1:a} \underline{$J(C_2)<h_2(\Omega)$}. 
	In this case, $C_1$ must be a Cheeger set of $\Omega \setminus C_2$. 
	Then, let $C_2'$ be a Cheeger set of $\Omega \setminus C_1$. It evidently satisfies $J(C_2')\leq J(C_2) < h_2(\Omega)$. By definition, $C_1$ is automatically a Cheeger set for $\Omega \setminus C_2'$. Hence, $(C_1, C_2')$ is a $1$-adjusted Cheeger couple of $\Omega$.
	\item \underline{$J(C_2)=h_2(\Omega)$}. As before, consider a Cheeger set $C_2'$ of $\Omega \setminus C_1$: if $J(C_2')<h_2(\Omega)$, we fall into the previous case. Otherwise, either $C_1$ is a Cheeger set of $\Omega \setminus C_2'$, and we are done; or this is not the case, and we take $C_1'$ to be a Cheeger set of $\Omega \setminus C_2'$. This last set will satisfy $J(C_1')<h_2(\Omega)$, and we fall again into the case \ref{thm:main1:1:a}. 
\end{enumerate}

\ref{thm:main1:2} Fix a natural number $k>2$ and suppose that the claim is true for every $k'<k$. We will show that it is true also for $k$. Let $(C_1, \dots,C_k)$ be a minimizer of $h_k(\Omega)$. 
We proceed as follows. Set $(E_1,\dots,E_k)=(C_1,\dots,C_k)$. 
\begin{enumerate}[label={\rm(\alph*)}]
	\item\label{thm:main1:2:a}
	Without loss of generality, we can suppose that $J(E_1) = \dots = J(E_m) = h_k(\Omega)$ for some $m \in \{1, \dots, k\}$, and $J(E_{m+1}), \dots, J(E_k) < h_k(\Omega)$ whenever $m<k$. 
	\item 
	If $m<k$, then, by the induction hypothesis, there exists a $1$-adjusted Cheeger $(k-m)$-tuple $(C_{m+1}', \dots, C_k')$ corresponding to $h_{k-m}(\Omega \setminus \bigcup_{j=1}^{m} E_j)$. 
	By definition, we will have $J(C_i') < h_k(\Omega)$ for every $i \in \{m+1,\dots,k\}$. Consider the new $k$-tuple 
	$$
	(E_1,\dots,E_m,E_{m+1},\dots,E_k)=(E_1,\dots,E_m,C_{m+1}',\dots,C_k').
	$$
	It is not hard to see that $E_i$ is a Cheeger set of $\Omega \setminus \bigcup_{j \neq i} E_j$ for any $i \in \{m+1,\dots,k\}$.
	\item 
	Set $i=1$.
	\item\label{thm:main1:2:d} If $i\leq m$ and $E_i$ is a Cheeger set of $\Omega \setminus \bigcup_{j \neq i} E_j$, then put $i=i+1$ and repeat step \ref{thm:main1:2:d}. 
	\item If $i\leq m$ and $E_i$ is not a Cheeger set of $\Omega \setminus \bigcup_{j \neq i} E_j$, then let $C'_i$ be such a Cheeger set, and consider the new $k$-tuple 
	$$
	(E_1,\dots,E_{i-1},E_i,E_{i+1},\dots,E_k)
	=
	(E_1,\dots,E_{i-1},C_i',E_{i+1},\dots,E_k).
	$$
	Observe that $J(E_i)<h_k(\Omega)$ and hence the number of sets $E_j$ such that $J(E_j) = h_k(\Omega)$ has been decreased by one unit. 	
	Go to step \ref{thm:main1:2:a}.
\end{enumerate}
At the end of this procedure, after a finite number of steps, we will obtain a $1$-adjusted Cheeger $k$-tuple of $\Omega$.

\ref{thm:main1:II} Fix some natural $n \geq 2$ and suppose that for any $n' < n$ and $k \geq n'$ there exists a $n'$-adjusted Cheeger $k$-tuple. 
Let us show now the existence of a $n$-adjusted Cheeger $k$-tuple for any $k \geq n$. 
Observe that the existence of a $n$-adjusted Cheeger $n$-tuple trivially follows from the existence of a $(n-1)$-adjusted Cheeger $n$-tuple. Therefore, it is enough to suppose that $k \geq n+1$.

\ref{thm:main1:1} First we show the existence of a $n$-adjusted Cheeger $(n+1)$-tuple. 
Let $(C_1,\dots,C_{n+1})$ be a minimizer of $h_{n+1}(\Omega)$. 
By the induction hypothesis with respect to the adjustment order, we can assume that $(C_1,\dots,C_{n+1})$ is a $(n-1)$-adjusted Cheeger $(n+1)$-tuple corresponding to $h_{n+1}(\Omega)$. 
Suppose that $(C_1,\dots,C_{n+1})$ is not $n$-adjusted. 
Therefore, there exists a $n$-subtuple of $(C_1,\dots,C_{n+1})$, say, $(C_1,\dots,C_{n})$, such that
$$
h_{n}\left(\Omega \setminus C_{n+1}\right) < \max\{J(C_1), \dots, J(C_n)\} \leq h_{n+1}(\Omega).
$$
Let us substitute $(C_1,\dots,C_n)$ by a $n$-adjusted Cheeger $n$-tuple $(C_1',\dots,C_n')$ corresponding to $h_n(\Omega\setminus C_{n+1})$. 
Then we get $J(C_1'),\dots,J(C_n') < J(C_{n+1}) = h_{n+1}(\Omega)$. 

Let us show that $(C_1',\dots,C_n',C_{n+1})$ is a $n$-adjusted Cheeger $(n+1)$-tuple of $\Omega$. 
Omit, for simplicity, the superscript $'$, and suppose, by contradiction, that there exists a $l$-tuple $(C_{i_1},\dots,C_{i_l})$ with some $l \in \{1,\dots,n\}$ such that 
$$
h_{l}\biggl(\Omega \setminus \bigcup_{j \neq i_1,\dots,i_l} C_{j} \biggr) 
< \max\{J(C_{i_1}),\dots, J(C_{i_l})\} \leq h_{n+1}(\Omega).
$$
Recalling that $(C_1,\dots,C_n)$ is a $n$-adjusted Cheeger $n$-tuple of $\Omega \setminus C_{n+1}$, we see that the $l$-tuple $(C_{i_1},\dots,C_{i_l})$ must necessarily contain $C_{n+1}$. 
Therefore, if $(E_{i_1},\dots,E_{i_l})$ is an arbitrary Cheeger $l$-tuple of the set $\Omega \setminus \bigcup_{j \neq i_1,\dots,i_l} C_{j}$, we get a contradiction to the definition of $h_{n+1}(\Omega)$, since \[J(E_{i_1}),\dots, J(E_{i_l}) < h_{n+1}(\Omega),\] and $J(C_i) < h_{n+1}(\Omega)$ for all $i \in \{1,\dots,n+1\}$ satisfying $i \neq i_j$ with $j \in \{1,\dots,l\}$.

\ref{thm:main1:2} Fix an arbitrary $k \geq n+1$ and suppose that for any $k' \in \{n,\dots,k-1\}$ there exists a $n$-adjusted Cheeger $k'$-tuple. Let us prove the existence of a $n$-adjusted Cheeger $k$-tuple. 

If $k=n+1$, this is easy to prove, so we can suppose that $k \geq n+2$. Let $(C_1,\dots,C_{k})$ be a $(n-1)$-adjusted Cheeger $k$-tuple of $\Omega$ which exists by the induction hypothesis with respect to the adjustment order stated in \ref{thm:main1:II}. 
Suppose that $(C_1,\dots,C_{k})$ is not $n$-adjusted, that is, there exits, say, $(C_1,\dots,C_n)$, such that 
$$
h_{n}\biggl(\Omega \setminus \bigcup_{j = n+1}^k C_{j}\biggr) < \max\{J(C_1), \dots, J(C_n)\} \leq h_{k}(\Omega).
$$
Let $(C_1',\dots,C_n')$ be a corresponding $n$-adjusted Cheeger $n$-tuple of $\Omega \setminus \bigcup_{j = n+1}^k C_{j}$. 
Note that \[J(C_1'),\dots,J(C_n') < h_k(\Omega),\] that is, there are at least $n$ elements of $(J(C_1'), \dots, J(C_n'), J(C_{n+1}), \dots, J(C_k))$ which are strictly smaller than $h_k(\Omega)$. 
We omit, for simplicity, the superscript $'$, and proceed as follows. 
Set $(E_1,\dots,E_k)=(C_1,\dots,C_k)$.
\begin{enumerate}[label={\rm(\alph*)}]
	\item\label{thm:main1:3:a}
	Without loss of generality, we can suppose that $J(E_1),\dots, J(E_m) < h_k(\Omega)$ for some $m \in \{n, \dots, k-1\}$, and $J(E_{m+1}) = \dots = J(E_k) = h_k(\Omega)$.
	\item\label{thm:main1:3:b} By the induction hypothesis, there exists a $n$-adjusted Cheeger $m$-tuple $(C_{1}', \dots, C_m')$ corresponding to $h_{m}(\Omega \setminus \bigcup_{j=m+1}^{k} E_j)$. 
	By definition, we will have $J(C_i') < h_k(\Omega)$ for every $i \in \{1,\dots,m\}$. Consider the new $k$-tuple 
	$$
	(E_1,\dots,E_m,E_{m+1},\dots,E_k)=(C_1',\dots,C_m',E_{m+1},\dots,E_k).
	$$
	It is not hard to see that any $l$-tuple $(E_{i_1},\dots,E_{i_l}) \subset (E_{1},\dots,E_m)$ is a $l$-adjusted Cheeger $l$-tuple of $\Omega \setminus \bigcup_{j \neq i_1,\dots,i_l} E_j$. 
	\item Assume that there exists $l \in \{1,\dots,n\}$ and a $l$-tuple $(E_{i_1},\dots,E_{i_l}) \subset (E_{1},\dots,E_k)$ such that $(E_{i_1},\dots,E_{i_l})$ is not a minimizer of $h_l\left(\Omega \setminus \bigcup_{j \neq i_1,\dots,i_l} E_j\right)$. Note that $(E_{i_1},\dots,E_{i_l}) \not\subset (E_{1},\dots,E_m)$, as it follows from step \ref{thm:main1:3:b}. Therefore, there exists $j \in \{1,\dots,l\}$ such that $E_{i_j} \in (E_{m+1},\dots,E_k)$, i.e., $J(E_{i_j}) = h_k(\Omega)$. 
	Let now $(C_{i_1}',\dots,C_{i_l}')$ be any minimizer of $h_l\left(\Omega \setminus \bigcup_{j \neq i_1,\dots,i_l} E_j\right)$. 
	Then we observe that $J(C_{i_j}') < h_k(\Omega)$. 
	Consider a new $k$-tuple $(E_1,\dots,E_k)$ obtained by replacing $(E_{i_1},\dots,E_{i_l})$ with $(C_{i_1}',\dots,C_{i_l}')$. 
	Hence, the number of sets $E_j$ such that $J(E_j) = h_k(\Omega)$ has been decreased at least by one unit. Go to step \ref{thm:main1:3:a}.
\end{enumerate}

At the end of this procedure, after a finite number of steps, we will obtain a $n$-adjusted Cheeger $k$-tuple of $\Omega$.
\end{proof}

\section[Regularity of adjusted Cheeger k-tuples]{Regularity of adjusted Cheeger $k$-tuples}\label{sec:regularity}

Let $(E_1,\dots,E_k) \in \mathcal{E}_k$ be a $1$-adjusted Cheeger $k$-tuple of $\Omega$. In this section we prove some regularity properties of $E_i$ by adapting the results obtained by Caroccia in \cite{caroccia}. 
Throughout this section we assume that $\Omega$ is a bounded open set.

We denote by $\mathcal{H}^{N-1}$ the $(N-1)$-dimensional Hausdorff measure. Given two Borel sets $E$ and $F$, we will write $E \approx F$ whenever $\mathcal{H}^{N-1}(E \bigtriangleup F)=0$, where $E \bigtriangleup F$ stands for the symmetric difference between $E$ and $F$. Given a Borel set $A$ and a set of finite perimeter $E$, we denote by $P(E;A)$ the perimeter of $E$ measured with respect to $A$. Recall also that, for the sake of simplicity, we write $P(E):=P(E;\R^N)$. We denote by $\partial^* E$ the \emph{reduced boundary} of $E$ (see \cite[p.\ 167]{maggi}). We recall that
\[ 
P(E;A) = \mathcal{H}^{N-1}(\partial^* E \cap A) \leq \mathcal{H}^{N-1}(\partial E \cap A).
\]
As an immediate consequence, if we denote the complement of a set $A$ in $\R^N$ by $A^c$, we have the basic relation
\[ P(E) = P(E;A) + P(E;A^c).\]
For $\alpha \in [0,1]$, we denote by $E^{(\alpha)}$ the set of points of Lebesgue density $s$, namely, 
\[E^{(\alpha)} = \left\{ x \in \R^N:~ \lim_{r \to 0^+} \frac{|E \cap B_r(x)|}{|B_r(x)|} = \alpha \right\}. \]
The set $E^{(1)}$ is called \emph{essential interior} of $E$, the set $E^{(0)}$ is the \emph{essential exterior} of $E$, and $\partial^e E = \R^N \setminus (E^{(0)} \cup E^{(1)})$ is the \emph{essential boundary} of $E$. By Federer's Theorem (cf.\ \cite[Theorem 16.2]{maggi}), we have
\[ 
\partial^* E \subset E^{(1/2)} \subset \partial^e E 
\quad \text{and} \quad 
\mathcal{H}^{N-1}(\partial^e E \setminus \partial^* E) = 0.
\]
This implies in particular that, for every set of finite perimeter $E$,
\begin{equation} \label{decompositionfederer} \R^N \approx E^{(0)}\cup E^{(1)} \cup \partial^* E. \end{equation}

\begin{lem} \label{maggidecomposition}
	Let $E$ and $F$ be sets of finite perimeter and $A$ be a Borel set. Then
	$$
	P(E \setminus F;A) + P(F \setminus E;A) \leq P(E;A)+P(F;A).
	$$
\end{lem}
\begin{proof}
	Define the set 
	\[\{\nu_E = -\nu_F\}:= \{ x \in \partial^* E \cap \partial^* F:~ \nu_E(x)=-\nu_F(x)\}.\] 
	By \cite[Theorem 16.3]{maggi} we have
	\begin{align*} P(E \setminus F;A) + P(F \setminus E;A) & = P(E;F^{(0)}\cap A) + P(F;E^{(1)}\cap A) + P(F;E^{(0)}\cap A) \\ & + P(E;F^{(1)}\cap A) + 2\mathcal{H}^{N-1}(\{\nu_E = -\nu_F\}\cap A) \\ & \leq P(E;F^{(0)}\cap A) + P(F;E^{(1)}\cap A) + P(F;E^{(0)}\cap A) \\ & + P(E;F^{(1)}\cap A) + 2\mathcal{H}^{N-1}(\partial^* E \cap \partial^* F \cap A) \\ & \leq P(E;F\cap A) + P(F;E\cap A), \end{align*}
	where we used relation \eqref{decompositionfederer}.
\end{proof}

\begin{lem} \label{strangedecomposition}
	Let $E$, $F$ and $L$ be sets of finite perimeter and $|E \cap F|=0$. Then
	\[ \partial^* (E \cap L) \cap \partial^*(F \cap L) \approx \partial^*E \cap \partial^* F \cap L^{(1)}.\]
\end{lem}
\begin{proof}
	By \cite[Theorem 16.3]{maggi} we have
	\begin{align*} 
	\partial^* (E \cap L) \cap \partial^*(F \cap L) & \approx [(L^{(1)} \cap \partial^* E) \cup (E^{(1)} \cap \partial^* L) \cup \{\nu_L=\nu_E\} ] \\ & \cap [(L^{(1)} \cap \partial^* F) \cup (F^{(1)} \cap \partial^* L) \cup \{\nu_L=\nu_F\} ].
	\end{align*}
	Using the fact that $\{\nu_L=\nu_E\}$ and $\{\nu_L=\nu_F\}$ are subsets of $\partial^* L \cap \partial^* E$ and $\partial^* L \cap \partial^* F$ respectively, and that $D^{(1)} \cap\partial^* D=\emptyset$ for every set of finite perimeter $D$, we obtain
	\begin{align*} 
	\partial^* (E \cap L) \cap \partial^*(F \cap L) & \approx (L^{(1)} \cap \partial^* E \cap \partial^* F) \cup (E^{(1)} \cap F^{(1)} \cap \partial^* L) \\ & \cup (E^{(1)} \cap \partial^* L \cap \{\nu_L = \nu_F\}) \cup (F^{(1)} \cap \partial^* L \cap \{\nu_L = \nu_E\}) \\ & \cup (\{\nu_L = \nu_E\} \cap \{\nu_L = \nu_F\}).
	\end{align*}
	Since $|E \cap F|=0$, it holds $E^{(1)}\subset F^{(0)}$ and $F^{(1)}\subset E^{(0)}$, and therefore
	$$
	E^{(1)} \cap F^{(1)} \cap \partial^* L = \emptyset
	\quad \text{and} \quad 
	\mathcal{H}^{N-1}(E^{(1)} \cap \partial^* F) = \mathcal{H}^{N-1}(F^{(1)} \cap \partial^* E) = 0
	$$
	by Federer's Theorem. Finally, if $x \in \partial^* E \cap \partial^* F$, then it is a point of density $\frac{1}{2}$ for both $E$ and $F$. Since $|E \cap F|=0$, it follows that $\nu_E(x)=-\nu_F(x)$ (see \cite[Exercice 12.9]{maggi}), and therefore
	\[ \{\nu_L=\nu_F = \nu_E \} = \emptyset.\]
	Hence we obtain the claim.
\end{proof}

\begin{definition}
	A set $M \subset \Omega$ has distributional mean curvature bounded from above by $g \in L^1_{loc}(\Omega)$ in $\Omega$ if there exists $r_0 > 0$ such that, for every $B_r \subset \subset \Omega$ with $r \in (0,r_0)$, and for every $L \subset M$ with $M \setminus L \subset \subset B_r$, it holds
	$$ 
	P(M;B_r) \leq P(L;B_r) + \int_{M \setminus L} g(x)\,dx.
	$$
\end{definition}

\begin{definition}
	Let $\Lambda$, $r_0 >0$. We say that a set of finite perimeter $E$ is $(\Lambda,r_0)$-perimeter minimizing in $\Omega$ if, for every $B_r \subset \subset \Omega$ with $r \in (0,r_0)$, and for every set of finite perimeter $F$ with $E \bigtriangleup F \subset \subset B_r$, it holds
	\[ P(E;B_r) \leq P(F;B_r) + \Lambda |E \bigtriangleup F|.\]
\end{definition}

\begin{lem} \label{lemma1regularity}
	Let $(E_1,\dots,E_k) \in \mathcal{E}_k$ be a $1$-adjusted Cheeger $k$-tuple of $\Omega$, and define
	\[ 
	M_i := \bigcup_{j \neq i} E_j.
	\]
	Then, each $M_i$ has distributional mean curvature bounded from above by $h_k(\Omega)$, namely, for every $L \subset M_i$ with $M_i \setminus L \subset \subset B_r$, where $B_r$ is a ball of radius $r < \frac{1}{h_k(\Omega)}$, it holds
	\[ 
	P(M_i;B_r) \leq P(L;B_r) + h_k(\Omega) |M_i \setminus L|.
	\]
\end{lem}
\begin{proof}
	Set $r_0=\frac{1}{h_k(\Omega)}$. Let $B_r \subset \subset \Omega$ be a ball of radius $r \in (0,r_0)$, and let $L \subset M_i$ be a set of finite perimeter such that $M_i \setminus L \subset \subset B_r$. Define $F_j := E_j \cap L$. By our choice of $r_0$, it holds $|F_j| > 0$ for every $j\neq i$; if this was not the case, then there would exist a set $E_j$ such that, up to negligible sets, $E_j \subset M_i \setminus L \subset \subset B_r$, and therefore
	\[ 
	h_k(\Omega) \geq \frac{P(E_j)}{|E_j|} \geq \frac{P(B_r)}{|B_r|} = \frac{N}{r} > \frac{N}{r_0} > h_k(\Omega), 
	\]
	a contradiction. Since $F_j \subset \Omega \setminus M_j$ for every $j \neq i$, and $E_j$ is a Cheeger set of $\Omega \setminus M_j$ (by $1$-adjustment assumption), it holds
	\[ 
	\frac{P(E_j)}{|E_j|} \leq \frac{P(F_j)}{|F_j|}
	\]
	and therefore
	\[ 
	\frac{P(E_j;B_r)+P(E_j;B_r^c)}{|E_j|} \leq \frac{P(F_j;B_r)+P(F_j;B_r^c)}{|E_j|-|E_j \setminus L|},
	\]
	which implies
	\[ 
	P(E_j;B_r) \leq P(F_j;B_r) + \frac{P(E_j)}{|E_j|}|E_j \setminus L| \leq P(F_j;B_r) + h_k(\Omega)|E_j \setminus L|.
	\]
	Using \cite[Lemma 3.3]{caroccia}, we obtain
	\begin{align*}
	P(M_i;B_r) & = \sum_{j \neq i} P(E_j;B_r) - \sum_{j,l\neq i,\,j\neq l} \mathcal{H}^{N-1}(\partial^* E_j \cap \partial^* E_l \cap B_r) \\ & \leq \sum_{j \neq i} \left(P(F_j;B_r) + h_k(\Omega)|E_j \setminus L|\right) - \sum_{j,l\neq i,\,j\neq l} \mathcal{H}^{N-1}(\partial^* E_j \cap \partial^* E_l \cap B_r) \\ & \leq \sum_{j \neq i} P(F_j;B_r) - \sum_{j,l\neq i,\,j\neq l} \mathcal{H}^{N-1}(\partial^* E_j \cap \partial^* E_l \cap B_r) +  h_k(\Omega)|M_i \setminus L|.
	\end{align*}
	Moreover, applying again \cite[Lemma 3.3]{caroccia}, we get
	\begin{align*} 
	P(L;B_r) & = P(M_i \cap L ;B_r)  \\ & = \sum_{j \neq i} P(F_j;B_r) - \sum_{j,l\neq i,\,j\neq l} \mathcal{H}^{N-1}(\partial^* F_j \cap \partial^* F_l \cap B_r) \\ & \geq \sum_{j \neq i} P(F_j;B_r) - \sum_{j,l\neq i,\,j\neq l} \mathcal{H}^{N-1}(\partial^* E_j \cap \partial^* E_l \cap B_r), 
	\end{align*}
	where we used Lemma \ref{strangedecomposition} and the fact that $M_i \cap L = L$. The above inequalities finally give
	\[ 
	P(M_i;B_r) \leq P(L;B_r) + h_k(\Omega) |M_i \setminus L|.
	\]
\end{proof}

\begin{rem}
Reasoning in a similar way, one can show that each Cheeger set $E_i$, $i \in \{1,\dots,k\}$, has distributional mean curvature bounded from above by $h_k(\Omega)$ in $\R^N$. 
\end{rem}

\begin{prop}
	Let $(E_1,\dots,E_k) \in \mathcal{E}_k$ be a $1$-adjusted Cheeger $k$-tuple of $\Omega$. Then, each $E_i$ is $(\Lambda,r_0)$-perimeter minimizing in $\Omega$ for $\Lambda = h_k(\Omega)$ and $r_0 = \frac{1}{h_k(\Omega)}$. 
\end{prop}
\begin{proof}
	As in Lemma \ref{lemma1regularity}, define
	\[ 
	M_i := \bigcup_{j \neq i} E_j.
	\]
	Let $r_0 = \frac{1}{h_k(\Omega)}$. Let $B_r \subset \subset \Omega$ be a ball of radius $r \in (0,r_0)$, and let $F_i$ be such that $F_i \bigtriangleup E_i \subset \subset B_r$. Define $D_i := F_i \setminus M_i$ and observe that, by the definition of a $1$-adjusted Cheeger $k$-tuple,
	\[ 
	\frac{P(E_i)}{|E_i|} \leq \frac{P(D_i)}{|D_i|}.
	\]
	Therefore, noting that $E_i = D_i = F_i$ in $B_r^c$, we get
	\begin{align*}
	\frac{P(E_i;B_r)+P(E_i;B_r^c)}{|E_i|} 
	 \leq \frac{P(D_i;B_r)+P(D_i;B_r^c)}{|F_i|-|F_i\cap M_i|} 
	= \frac{P(D_i;B_r)+P(E_i;B_r^c)}{|E_i|+|F_i\cap B_r| - |E_i \cap B_r| -|F_i\cap M_i|},
	\end{align*}
	and hence
	\[ 
	P(E_i;B_r)|E_i| \leq P(D_i;B_r)|E_i| + P(E_i)(|E_i \cap B_r| +|F_i\cap M_i| - |F_i\cap B_r|).
	\]
	Observing that $F_i \cap M_i \subset F_i \setminus E_i$, we obtain
	\begin{align*}
	|E_i \cap B_r| +|F_i\cap M_i| - |F_i\cap B_r| 
	&\leq 
	|E_i \cap B_r| +|F_i \setminus E_i| - |F_i\cap B_r| \\
	&= 
	|E_i \cap B_r| +|F_i  \cap B_r| - |F_i \cap E_i \cap B_r| - |F_i\cap B_r| \\
	&= |E_i \cap B_r| - |F_i \cap E_i \cap B_r| = |E_i \setminus F_i|,
	\end{align*}
	and hence
	\begin{equation}\label{lem37:1}
	P(E_i;B_r) \leq P(D_i;B_r) + \frac{P(E_i)}{|E_i|} |E_i \setminus F_i| \leq P(D_i;B_r) + h_k(\Omega) |E_i \setminus F_i|.
	\end{equation}
	On the other hand, using the fact that $M_i \setminus F_i \subset M_i$ and $M_i \setminus (M_i \setminus F_i) = M_i \cap F_i \subset \subset B_r$, by Lemma \ref{lemma1regularity} we get
	\begin{align*} 
	P(M_i;B_r) \leq P(M_i \setminus F_i;B_r) + h_k(\Omega) |M_i \cap F_i| \leq P(M_i \setminus F_i;B_r) + h_k(\Omega) |F_i \setminus E_i|.
	\end{align*}
	Therefore, by Lemma \ref{maggidecomposition},
	\begin{equation}\label{lem37:2}
	P(D_i;B_r) \leq P(F_i;B_r) + P(M_i;B_r) - P(M_i \setminus F_i;B_r) \leq P(F_i;B_r) + h_k(\Omega) |F_i \setminus E_i|.
	\end{equation}
	Finally, combining \eqref{lem37:1} and \eqref{lem37:2}, we obtain
	\[
	P(E_i;B_r) \leq P(F_i;B_r) + h_k(\Omega)|E_i \bigtriangleup F_i|,
	\]
	which proves that $E_i$ is $(\Lambda,r_0)$-perimeter minimizing in $\Omega$ with $\Lambda = h_k(\Omega)$ and $r_0 = \frac{1}{h_k(\Omega)}$. 
\end{proof}

\begin{thm}\label{thm:interiorregularity}
	Let $(E_1,\dots,E_k) \in \mathcal{E}_k$ be a $1$-adjusted Cheeger $k$-tuple of $\Omega$. Then, for each $i \in \{1,\dots,k\}$, the following assertions hold:
	\begin{enumerate}[label={\rm(\roman*)}]
		\item\label{thm:interiorregularity:i} $\partial^* E_i \cap \Omega$ is of class $C^{1,\gamma}$ for every $\gamma \in (0,\frac{1}{2})$.
		\item\label{thm:interiorregularity:ii} The set $(\partial E_i \setminus \partial^* E_i) \cap \Omega$ has Hausdorff dimension at most $N-8$.
		\item\label{thm:interiorregularity:iii} If $N \leq 7$, then $\partial E_i \cap \Omega$ is of class $C^{1,\gamma}$ for every $\gamma \in (0,\frac{1}{2})$.
		\item\label{thm:interiorregularity:iv} Suppose that $\mathcal{H}^{N-1}(\partial \Omega)<+\infty$. Then there exists an open set $\widetilde{E}_i$ such that $|E_i \bigtriangleup \widetilde{E}_i| = 0$. Moreover, $(\widetilde{E}_1,\dots,\widetilde{E}_k)$ is a $1$-adjusted Cheeger $k$-tuple.
		\item\label{thm:interiorregularity:v} Suppose that $\Omega$ has finite perimeter. Then $\partial E_i \cap \Omega$ can meet $\partial^* \Omega$ only in a tangential way, that is, if $x \in \partial^* \Omega \cap \partial E_i$, then $x \in \partial^* E_i$, and $\nu_\Omega(x)=\nu_{E_i}(x)$.
	\end{enumerate}
\end{thm}
\begin{proof}
	The proof of \ref{thm:interiorregularity:i} and \ref{thm:interiorregularity:ii} follows from the fact that each $E_i$ is $(\Lambda,r_0)$-perimeter minimizing in $\Omega$, and from classical regularity results, cf.\ \cite[Theorems 21.8 and 28.1]{maggi}. Assertion \ref{thm:interiorregularity:iii} easily follows from \ref{thm:interiorregularity:i} and \ref{thm:interiorregularity:ii}. Let us now prove \ref{thm:interiorregularity:iv}. Since $\mathcal{H}^{N-1}(\partial \Omega)<+\infty$, by \ref{thm:interiorregularity:ii} we have that the topological boundary of each $E_i$ has Hausdorff dimension $N-1$. If we define $\widetilde{E}_i := E_i \setminus \partial E_i$, then $\widetilde{E}_i$ is an open set such that $P(\widetilde{E}_i)=P(E_i)$ and $|\widetilde{E}_i|=|E_i|$. Therefore, $(\widetilde{E}_1,\dots,\widetilde{E}_k)$ is a $1$-adjusted Cheeger $k$-tuple of $\Omega$. Finally, assertion \ref{thm:interiorregularity:v} can be proven as in \cite[Appendix A]{leonardipratelli}.
\end{proof}

\begin{rem}
If $\Omega$ has a boundary of class $C^1$, and $N \leq 7$, then Theorem \ref{thm:interiorregularity} implies that the boundary of each $1$-adjusted Cheeger set $E_i$ is of class $C^1$ as well. Concerning the boundary regularity, we also refer to \cite[Theorem 3]{gonzalezmassaritamanini}.
\end{rem}

\begin{prop}\label{prop:approximation}
Suppose that $\mathcal{H}^{N-1}(\partial \Omega \setminus \partial^* \Omega)=0$. Let $(E_1,\dots,E_k) \in \mathcal{E}_k$ be a $1$-adjusted Cheeger $k$-tuple of $\Omega$ such that every $E_i$ is open. Then, for every $i \in \{1,\dots,k\}$, there exists a sequence $\{E_{i,m}\}_{m \in \mathbb{N}}$ such that:
 \begin{enumerate}[label={\rm(\roman*)}]
  \item\label{prop:approximation:1} $E_{i,m} \subset \subset E_i$ for every $m \in \mathbb{N}$;
  \item\label{prop:approximation:2} $\partial E_{i,m}$ is smooth for every $m \in \mathbb{N}$;
  \item\label{prop:approximation:3} $E_{i,m} \to E_i$ in $L^1(\Omega)$ as $m \to +\infty$;
  \item\label{prop:approximation:4} $P(E_{i,m}) \to P(E_i)$ as $m \to +\infty$.
 \end{enumerate}
\end{prop}
\begin{proof}
 By assertion \ref{thm:interiorregularity:ii} of Theorem \ref{thm:interiorregularity}, each $E_i$ satisfies $\mathcal{H}^{N-1}((\partial E_i \cap \Omega)\setminus \partial^* E_i)=0$. Moreover, by assertion \ref{thm:interiorregularity:v} of Theorem \ref{thm:interiorregularity},  $\partial^* \Omega \cap \partial E_i \subset \partial^* E_i$. Therefore
 \begin{align*} 
 \mathcal{H}^{N-1}(\partial E_i \setminus \partial^* E_i) & = \mathcal{H}^{N-1}((\partial E_i \cap \Omega )\setminus \partial^* E_i) + \mathcal{H}^{N-1}((\partial E_i \cap \partial \Omega)\setminus \partial^* E_i) \\ & \leq \mathcal{H}^{N-1}((\partial E_i \cap \partial \Omega)\setminus (\partial^* \Omega \cap \partial E_i))  = 0.
 \end{align*}
 Therefore, we can apply the approximation result of \cite[Theorem 1.1]{schmidt} to obtain the desired claims. 
\end{proof}

\section[Properties of Cheeger k-tuples]{Properties of Cheeger $k$-tuples}\label{sec:properties}
In this section we prove some qualitative properties of Cheeger $k$-tuples $(E_1,\dots,E_k) \in \mathcal{E}_k$ of $\Omega$. 
Throughout this section we assume that $\Omega$ is a bounded open set.
First we introduce several notations.
\begin{definition}
	The \textit{free boundary} of $E_i$ is 
	$$
	\partial_f E_i := \partial E_i \cap \biggl( \Omega \setminus \bigcup_{j \neq i} \overline{E_j} \biggr).
	$$ 
\end{definition}
\begin{definition}
	The \textit{contact surface} between $E_i$ and $E_j$ (for $i \neq j$) is
	$$
	\partial(E_i E_j) := \partial E_i \cap
	\partial E_j \cap \Omega.
	$$
\end{definition}
\begin{definition}
	The \textit{boundary surface} of $E_i$ is the contact surface between $E_i$ and $\partial \Omega$, that is,
	$$
	\partial_b E_i := \partial E_i \cap \partial \Omega.
	$$
\end{definition}

If $(E_1,\dots,E_k)$ is a $1$-adjusted Cheeger $k$-tuple of $\Omega$, then, in view of Theorem \ref{thm:interiorregularity}, the following decomposition takes place:
$$
\partial E_i = \bigcup_{j \neq i} \partial(E_i E_j) 
\cup \partial_f E_i
\cup \partial_b E_i.
$$
We will denote by $\partial_f^* E_i$ and $\partial^* (E_iE_j)$ the reduced part of $\partial_f E_i$ and $\partial (E_iE_j)$, respectively.

The following results are a consequence of \cite[Theorem 2]{gonzalezmassaritamanini}.
\begin{prop}\label{prop:meancurvature1}
	Let $(E_1,\dots,E_k) \in \mathcal{E}_k$ be a $1$-adjusted Cheeger $k$-tuple of $\Omega$. 
	Then the mean curvature of $\partial_f^* E_i$ (measured from the inside of $E_i$) is a constant equal to $\frac{h_1(E_i)}{N-1}$ for every $i \in \{1,\dots,k\}$.
\end{prop}

\begin{prop}\label{prop:meancurvature2}
	Let $(E_1,\dots,E_k) \in \mathcal{E}_k$ be a $2$-adjusted Cheeger $k$-tuple of $\Omega$. 
	Then the mean curvature of $\partial^* (E_iE_j)$ (for $i \neq j$) is constant.
\end{prop}

\begin{rem}
	Note that the result of Proposition \ref{prop:meancurvature2} can be false for $1$-adjusted Cheeger $k$-tuples, see Fig.~\ref{fig:Fig1}, c).
\end{rem}

Hereinafter, for a $2$-adjusted Cheeger $k$-tuple $(E_1,\dots,E_k)$ of $\Omega$, we denote by $c_{ij}$ the mean curvature of $\partial^* (E_i E_j)$ measured from the inside of $E_i$.

\begin{prop}\label{prop1}
	Let $E_i$ and $E_j$ be any elements of a $2$-adjusted Cheeger $k$-tuple $(E_1,\dots,E_k)$ of $\Omega$ such that $\partial^*(E_iE_j) \neq \emptyset$. Assume that $c_{ij} \geq 0$. 
	Then $h_1(E_i) \geq h_1(E_j)$ and $h_1(E_i) \geq c_{ij}$. Moreover, if $h_1(E_i) > h_1(E_j)$, then $h_1(E_i) = c_{ij}$.
\end{prop}
\begin{proof}
	Let us show first that $h_1(E_i) \geq c_{ij}$. 
	If $c_{ij}=0$, then the result is obvious. Assume that $c_{ij} > 0$.
	Suppose, by contradiction, that $h_1(E_i) < c_{ij}$.
	Take any $x \in \partial^*(E_iE_j)$. By Theorem \ref{thm:interiorregularity} \ref{thm:interiorregularity:i}, there is a neighborhood of $x$ where $\partial^*(E_iE_j)$ can be described as a graph of a function $u \in C^{1,\gamma}(\omega)$, where $\omega$ is an open subset of $\mathbb{R}^{N-1}$.
	Let us take any $\varphi \in C_0^\infty(\omega)\setminus\{0\}$ such that $\varphi \geq 0$, and let us perturb $\partial^*(E_iE_j)$ by $-\varepsilon \varphi$ for $\varepsilon > 0$ small enough, so that we obtain two new sets $E_i^\varepsilon$ and $E_j^\varepsilon$. The quotient between the perimeter and the volume of $E_i^\varepsilon$ satisfies the relation
	\begin{align*}
	\notag
	\frac{P(E_i^\varepsilon)}{|E_i^\varepsilon|} 
	&= \frac{\left(P(E_i) - \int_\omega \sqrt{1+|\nabla u|^2} \, dx\right) + \int_\omega \sqrt{1+|\nabla(u-\varepsilon \varphi)|^2} \, dx}{\left(|E_i| - \int_{\omega} u \, dx \right) + \int_{\omega} (u- \varepsilon \varphi) \, dx} \\
	\label{eq:prop1}
	&\approx 
	\frac{P(E_i) - \varepsilon \int_\omega\frac{\nabla u \nabla \varphi}{\sqrt{1+|\nabla u|^2}}\, dx}{|E_i| - \varepsilon \int_{\omega} \varphi \, dx}.
	\end{align*}
	Since $\partial^*(E_iE_j)$ has a constant mean curvature $c_{ij}$, we have
	$$
	\int_\omega\frac{\nabla u \nabla \varphi}{\sqrt{1+|\nabla u|^2}}\, dx = c_{ij} \int_\omega \varphi \, dx > 0.
	$$
	Therefore, recalling that $\frac{P(E_i)}{|E_i|} = h_1(E_i) < c_{ij}$, we easily deduce that
	\begin{equation}\label{eq:min:1}
	\frac{P(E_i^\varepsilon)}{|E_i^\varepsilon|} \approx
	\frac{P(E_i) - \varepsilon c_{ij} \int_\omega \varphi \, dx}{|E_i| - \varepsilon \int_{\omega} \varphi \, dx} < \frac{P(E_i)}{|E_i|}.
	\end{equation}	
	On the other hand, applying the same perturbation $-\varepsilon \varphi$ to $E_j$, we see that $P(E_j^\varepsilon)$ decreases and $|E_j^\varepsilon|$ increases, which implies that 
	\begin{equation}\label{eq:min:2}
	\frac{P(E_j^\varepsilon)}{|E_j^\varepsilon|} < \frac{P(E_j)}{|E_j|}.
	\end{equation}
	However, \eqref{eq:min:1} and \eqref{eq:min:2} contradict the fact that $(E_i, E_j)$ is a minimizer of 
	$h_2\left(\Omega \setminus \bigcup_{m \neq i,j} E_m\right)$, since $(E_1,\dots,E_k)$ is a $2$-adjusted Cheeger $k$-tuple of $\Omega$. Therefore, $h_1(E_i) \geq c_{ij}$.
	
	Let us show now that $h_1(E_i) \geq h_1(E_j)$. Suppose, by contradiction, that $h_1(E_i) < h_1(E_j)$. 
	Then $h_2\left(\Omega \setminus \bigcup_{m \neq i,j} E_m\right) = h_1(E_j)$. However, applying the perturbation $-\varepsilon \varphi$ as above, we see from \eqref{eq:min:2} that
	\begin{equation}\label{eq:min:3}
	h_2\biggl(\Omega \setminus \bigcup_{m \neq i,j} E_m\biggr) \leq \max\left\{\frac{P(E_i^\varepsilon)}{|E_i^\varepsilon|}, \frac{P(E_j^\varepsilon)}{|E_j^\varepsilon|}\right\} = \frac{P(E_j^\varepsilon)}{|E_j^\varepsilon|} < h_1(E_j)
	\end{equation}
	for sufficiently small $\varepsilon > 0$. 	A contradiction.
	
	Let us show finally that if $h_1(E_i) > h_1(E_j)$, then $h_1(E_i) = c_{ij}$. Suppose, by contradiction, that $h_1(E_i) > c_{ij}$. Using the perturbation argument as above, but with a positive perturbation $\varepsilon \varphi$ we get
	\begin{align*}
	\frac{P(E_i^\varepsilon)}{|E_i^\varepsilon|} & =\frac{\left(P(E_i) - \int_\omega \sqrt{1+|\nabla u|^2} \, dx\right) + \int_\omega \sqrt{1+|\nabla(u+\varepsilon \varphi)|^2} \, dx}{\left(|E_i| - \int_{\omega} u \, dx \right) + \int_{\omega} (u+ \varepsilon \varphi) \, dx} \\ & 
	\approx \frac{P(E_i) + \varepsilon c_{ij} \int_\omega \varphi \, dx}{|E_i| + \varepsilon \int_{\omega} \varphi \, dx} < \frac{P(E_i)}{|E_i|}
	\end{align*}
	for sufficiently small $\varepsilon > 0$. However, we again get a contradiction as in \eqref{eq:min:3}, since $(E_1,\dots,E_k)$ is a $2$-adjusted Cheeger $k$-tuple of $\Omega$.
\end{proof}

\section{Spectral minimal partitions}\label{sec:spectral_minimal_partitions}
In this section we show that $h_k(\Omega)$ can be characterized as a limit of the energy $\mathfrak{L}(p;\Omega)$ of the \textit{spectral minimal partition} of $\Omega$ with respect to the $p$-Laplacian as $p \to 1$. 
Along the whole section, we will assume that $\Omega \subset \mathbb{R}^N$ is a bounded open set. 

Let $E$ be a measurable subset of $\Omega$. We define the first eigenvalue of the $p$-Laplacian on $E$ as
$$
\lambda_1(p;E) := \inf\left\{\frac{\int_E |\nabla u|^p \, dx}{\int_E |u|^p \, dx}:~ u \in W_0^{1,p}(\Omega) \setminus \{0\},~ u=0 \text{ a.e.\ on } \Omega \setminus E \right\},
$$
and we put $\lambda_1(p;E) = +\infty$ whenever the admissible set of functions is empty. 
Since the constraint $u=0$ a.e.\ on $\Omega \setminus E$ is weakly compact, we get the existence of a minimizer (i.e., eigenfunction) of $\lambda_1(p;E)$. 
Note that if $E$ is open and $\partial E$ is continuous, then $\lambda_1(p;E)$ is the usual first eigenvalue of the $p$-Laplacian on $E$, cf.\ \cite[Theorem 5.29]{adams}.

Let us define $\mathfrak{L}(p;\Omega)$ as
\begin{equation*}
\mathfrak{L}_k(p;\Omega) := \inf_{(E_1,\dots,E_k) \in \mathcal{E}_k} \max_{i = 1,\dots,k} \lambda_1(p;E_i).
\end{equation*}
A $k$-tuple $(E_1,\dots,E_k) \in \mathcal{E}_k$ which delivers a minimum to $\mathfrak{L}_k(p;\Omega)$ is called a $k$-th spectral minimal partition of $\Omega$. 
We start with the existence result for $\mathfrak{L}(p;\Omega)$. 
\begin{prop}\label{prop:properties}
	$\mathfrak{L}_k(p;\Omega)$ is attained for any $k \in \mathbb{N}$ and $p > 1$, that is, there exists a $k$-th spectral minimal partition of $\Omega$.	
\end{prop}
\begin{proof}
	First we introduce the following auxiliary minimization problem:
	\begin{equation*}\label{def:Lktilde}
	\widetilde{\mathfrak{L}}_k(p;\Omega) := 
	\inf\left\{ 
	\mathcal{L}(u_1, \dots, u_k):~
	u_i \in W_0^{1,p}(\Omega) \setminus \{0\},~ u_i \cdot u_j = 0 \text{ a.e.\ on } \Omega \text{ for } i \neq j
	\right\},
	\end{equation*}
	where the functional $\mathcal{L}: (W_0^{1,p}(\Omega))^k \to \mathbb{R}$ is defined by 
	$$
	\mathcal{L}(u_1, \dots, u_k) := \max\left\{ \frac{\int_\Omega |\nabla u_1|^p \, dx}{\int_\Omega |u_1|^p \, dx}, \dots, \frac{\int_\Omega |\nabla u_k|^p \, dx}{\int_\Omega |u_k|^p \, dx} \right\}.
	$$
	We claim that for any $p>1$ and $k \in \mathbb{N}$ there exists a minimizer $(\varphi_1, \dots, \varphi_k) \in (W_0^{1,p}(\Omega))^k$ for $\widetilde{\mathfrak{L}}_k(p;\Omega)$, where each $\varphi_i \not\equiv 0$ in $\Omega$. 
	Let $(\varphi_1^n, \dots, \varphi_k^n) \in (W_0^{1,p}(\Omega))^k$, $n \in \mathbb{N}$, be a minimizing sequence for $\widetilde{\mathfrak{L}}_k(p;\Omega)$. Let $C > 0$ be such that $\mathcal{L}(\varphi_1^n, \dots, \varphi_k^n) < C$ for all $n \in \mathbb{N}$.
	Due to the $0$-homogeneity of $\mathcal{L}$, we can assume that $\|\varphi_i^n\|_{L^p(\Omega)} = 1$ for each $i \in \{1,\dots,k\}$ and $n \in \mathbb{N}$. 
	Therefore, we obtain that $\|\nabla \varphi_i^n\|_{L^p(\Omega)} < C$ for each $i \in \{1,\dots,k\}$ and $n \in \mathbb{N}$. Hence, by a diagonal argument, we can find a subsequence $(\varphi_1^{n_j}, \dots, \varphi_k^{n_j})$, $j \in \mathbb{N}$, and a vector $(\varphi_1, \dots, \varphi_k) \in (W_0^{1,p}(\Omega))^k$ such that each $\varphi_i^{n_j} \to \varphi_i$ weakly in $W_0^{1,p}(\Omega)$, strongly in $L^p(\Omega)$, and almost everywhere in $\Omega$ as $j \to +\infty$. 
	Thus, since $\|\varphi_i^n\|_{L^p(\Omega)} = 1$, we conclude that $\|\varphi_i\|_{L^p(\Omega)} = 1$, that is, $\varphi_i \not\equiv 0$ in $\Omega$. Moreover, due to a.e.-convergence, we deduce that $\varphi_i \cdot \varphi_j = 0$ a.e.\ on $\Omega$ for $i \neq j$. 
	Therefore, $(\varphi_1, \dots, \varphi_k)$ is an admissible vector for $\widetilde{\mathfrak{L}}_k(p;\Omega)$. Finally, considering $m \in \{1,\dots,k\}$ such that 
	$$
	\max_{i = 1,\dots,k} \|\nabla \varphi_i\|_{L^p(\Omega)}^p = \|\nabla \varphi_m\|_{L^p(\Omega)}^p
	$$ 
	and noting that 
	\begin{align*}
	\widetilde{\mathfrak{L}}_k(p;\Omega) 
	&\leq \mathcal{L}(\varphi_1, \dots, \varphi_k) 
	=
	\|\nabla \varphi_m\|_{L^p(\Omega)}^p \leq \liminf_{j \to +\infty} \|\nabla \varphi_m^{n_j}\|_{L^p(\Omega)}^p \\
	&\leq 
	\liminf_{j \to +\infty} \max_{i = 1,\dots,k} \|\nabla \varphi_i^{n_j}\|_{L^p(\Omega)}^p = 
	\liminf_{j \to +\infty} \mathcal{L}(\varphi_1^{n_j}, \dots, \varphi_k^{n_j})
	= \widetilde{\mathfrak{L}}_k(p;\Omega),
	\end{align*}
	we conclude that $(\varphi_1, \dots, \varphi_k)$ is a minimizer of $\widetilde{\mathfrak{L}}_k(p;\Omega)$. 
	
	Evidently, we can assume that each $\varphi_i \geq 0$ a.e.\ on $\Omega$. Let us denote $E_i = \{x \in \Omega: \varphi_i(x) > 0\}$ for $i \in \{1,\dots,k\}$. Obviously, each $|E_i|>0$. Moreover, since $\varphi_i \cdot \varphi_j = 0$ a.e.\ on $\Omega$ for $i \neq j$, we get $E_i \cap E_j = \emptyset$ a.e.\ on $\Omega$. 
	Therefore, we can assume that $(E_1,\dots,E_k) \in \mathcal{E}_k$ and hence $\mathfrak{L}_k(p;\Omega) \leq \widetilde{\mathfrak{L}}_k(p;\Omega)$. 
	The opposite inequality $\widetilde{\mathfrak{L}}_k(p;\Omega) \leq \mathfrak{L}_k(p;\Omega)$ is obvious by the definitions. 
	Thus, we conclude that $\mathfrak{L}_k(p;\Omega) = \widetilde{\mathfrak{L}}_k(p;\Omega)$ and a minimizer of $\widetilde{\mathfrak{L}}_k(p;\Omega)$ defines a $k$-th spectral minimal partition of $\Omega$.	
\end{proof}

	\begin{rem}
		Evidently, $\mathfrak{L}_1(p;\Omega) = \lambda_1(p;\Omega)$. Moreover, using the variational characterization of $\lambda_2(p;\Omega)$ (see, e.g., \cite{garciaazoreroperal} or \cite{drabrob1999}), it is not hard to obtain that $\mathfrak{L}_2(p;\Omega) = \lambda_2(p;\Omega)$.
	\end{rem}

	The result of Proposition \ref{prop:properties} can be refined in the following way. Let us recall several definitions. Under  \textit{$p$-capacity} of a measurable set $E \subset \Omega$ we mean
	$$
	\capp(E) = \inf\left\{ \int_\Omega |\nabla u|^p \, dx:~ u \in W_0^{1,p}(\Omega),~ u \geq 1 \text{ a.e.\ in a neighborhood of } E\right\}.
	$$
	A subset $E$ of $\Omega$ is called \textit{$p$-quasi-open} if for every $\varepsilon > 0$ there exists an open subset $E_\varepsilon$ of $\Omega$, such that $\capp(E_\varepsilon \bigtriangleup E) < \varepsilon$. If some abstract property $G(x)$ is satisfied for all $x \in \Omega$ except, possibly, for elements of a set $Z \subset \Omega$ with $\capp(Z) = 0$, we say that $G(x)$ is satisfied \textit{$p$-quasi-everywhere} on $\Omega$, or $p$-q.e.\ on $\Omega$, for short. Note that $\capp(Z) = 0$ implies $\mathcal{H}^{N-1}(Z) = 0$ and hence $|Z| = 0$, cf.\ \cite[Section 4.7.2, Theorem 4]{evans}.
	
	Consider now the subclass $\mathcal{E}_k^* \subset \mathcal{E}_k$ of $k$-tuples $(E_1,\dots,E_k)$ where each $E_i$ is $p$-quasi-open, and define the quantity
	\begin{equation*}\label{def:Lk}
	\mathfrak{L}_k^*(p;\Omega) := \inf_{(E_1,\dots,E_k) \in \mathcal{E}_k^*} \max_{i = 1,\dots,k} \lambda_1(p;E_i).
	\end{equation*}
	\begin{lem}\label{lem:pqo}
		$\mathfrak{L}_k^*(p;\Omega) = \mathfrak{L}_k(p;\Omega)$ for any $k \in \mathbb{N}$ and $p>1$, that is, there exists a $p$-quasi-open $k$-th spectral minimal partition of $\Omega$.
	\end{lem}
	\begin{proof}
	We prove that $\mathfrak{L}_k^*(p;\Omega) = \widetilde{\mathfrak{L}}_k(p;\Omega)$ where $\widetilde{\mathfrak{L}}_k(p;\Omega)$ is defined in Proposition \ref{prop:properties}. 
	Let $(\varphi_1,\dots\varphi_k)$ by a minimizer of $\widetilde{\mathfrak{L}}_k(p;\Omega)$.
	We can assume that each $\varphi_i \geq 0$ a.e.\ on $\Omega$.
	Moreover, we can identify each $\varphi_i$ with its $p$-quasi-continuous representative (cf.\ \cite[Section~3.3, Theorem~3.3.3]{ziemer}), that is, we can assume that for every $\varepsilon > 0$ there exists a continuous function $\varphi_i^\varepsilon$ such that $\capp(\{\varphi_i \neq \varphi_i^\varepsilon\}) < \varepsilon$. Therefore, denoting $E_i = \{x \in \Omega: \varphi_i(x) > 0\}$, we see that each $E_i$ is a $p$-quasi-open subset of $\Omega$. 
	Moreover, since $\varphi_i \cdot \varphi_j = 0$ a.e.\ on open $\Omega$ for $i \neq j$, we derive from \cite{heinonen} that $\varphi_i \cdot \varphi_j = 0$ $p$-q.e.\ on $\Omega$ and hence $E_i \cap E_j = \emptyset$ $p$-q.e.\ on $\Omega$.
	Therefore, we can assume that $(E_1, \dots, E_k) \in \mathcal{E}_k^*$ and hence $\mathfrak{L}_k^*(p;\Omega) \leq \widetilde{\mathfrak{L}}_k(p;\Omega)$. 
	The opposite inequality $\widetilde{\mathfrak{L}}_k(p;\Omega) \leq \mathfrak{L}_k^*(p;\Omega)$ is trivial. 	
	Thus, we proved that $\mathfrak{L}_k^*(p;\Omega) = \mathfrak{L}_k(p;\Omega)$.
	\end{proof}
	
\medskip
Now we prove the main result of this section. 
\begin{thm}\label{thm:1}
	$h_k(\Omega) = \lim\limits_{p \to 1} \mathfrak{L}_k(p;\Omega)$ for any $k \in \mathbb{N}$.
\end{thm}
\begin{proof}
We follow the strategy of \cite[Section~5]{Parini}. 
Let us fix an arbitrary $k \in \mathbb{N}$. We show first that 
\begin{equation}\label{eq:limit1}
\limsup\limits_{p \to 1} \mathfrak{L}_k(p;\Omega) \leq h_k(\Omega).
\end{equation}
Let $(E_1, \dots, E_k) \in \mathcal{E}_k$ be a $1$-adjusted Cheeger $k$-tuple of $\Omega$ obtained by Theorem \ref{thm:main1}. 
By Proposition \ref{prop:approximation}, we can approximate each $E_i$ by a sequence of sets of finite perimeter $\{E_{i,m}\}_{m \in \mathbb{N}}$ with smooth boundary, compactly contained in $E_i$, such that
\[ 
\frac{P(E_{i,m})}{|E_{i,m}|}\leq h_k(\Omega) + \frac{1}{m}.
\]
For $\varepsilon > 0$, define
$$
E_{i,m}^\varepsilon := \{x \in \mathbb{R}^N:~ \text{dist}(x, E_{i,m}) \leq \varepsilon\}.
$$
Let $m \in \mathbb{N}$ be fixed. If $\varepsilon > 0$ is sufficiently small in order to have $E_{i,m}^\varepsilon \subset \subset E_i$ for each $i \in \{1,\dots, k\}$, we define functions $\{v_{i,m}\}_{i=1}^k \subset W_0^{1,p}(\Omega)$ such that $v_{i,m} = 1$ on $E_{i,m}$, $v_{i,m} = 0$ on $\Omega \setminus E_{i,m}^\varepsilon$ and $|\nabla v_{i,m}|  = \varepsilon^{-1}$ on $E_{i,m}^\varepsilon \setminus E_{i,m}$. 
Then we have 
$$
\mathfrak{L}_k(p;\Omega) 
\leq \max_{i=1,\dots,k} \lambda_1(p;E_i)
\leq \max_{i=1,\dots,k} \frac{\int_\Omega |\nabla v_{i,m}|^p \, dx}{\int_\Omega |v_{i,m}|^p \, dx} 
\leq \max_{i=1,\dots,k} \frac{\varepsilon^{-p}|{E}_{i,m}^\varepsilon \setminus E_{i,m}|}{|E_{i,m}|}.
$$
Noting that $|{E}_{i,m}^\varepsilon \setminus E_{i,m}| = \varepsilon P(E_{i,m}) + o(\varepsilon)$, where $o(\varepsilon)/\varepsilon \to 0$ as $\varepsilon \to 0$ (cf. \cite[Corollary~1]{Ambrosio}), and that the volume of the sets $E_{i,m}$ is uniformly bounded from below, we conclude that
$$
\mathfrak{L}_k(p;\Omega) 
\leq \max_{i=1,\dots,k} \frac{\varepsilon^{1-p}P(E_{i,m}) + \varepsilon^{-p} o(\varepsilon)}{|E_{i,m}|} 
\leq \frac{1}{\varepsilon^{p-1}}\left( h_k(\Omega) + \frac{1}{m}\right) + \frac{o(\varepsilon)}{\varepsilon^{p}}.
$$
Therefore, for each $m \in \mathbb{N}$ we have
$$
\limsup_{p \to 1}\mathfrak{L}_k(p;\Omega) 
\leq h_k(\Omega) + \frac{1}{m} + \frac{o(\varepsilon)}{\varepsilon}.
$$
Letting first $\varepsilon \to 0$, and then $m \to +\infty$, we derive \eqref{eq:limit1}.

Let us show now that
\begin{equation}\label{eq:limit2}
\liminf\limits_{p \to 1} \mathfrak{L}_k(p;\Omega) \geq h_k(\Omega).
\end{equation}
From Proposition~\ref{prop:properties} we know that for any $p>1$ there exists a spectral minimal partition $(E_1,\dots,E_k) \in \mathcal{E}_k$ for $\Omega$, that is, $\mathfrak{L}_k(p;\Omega) = \max\limits_{i=1,\dots,k} \lambda_1(p;E_i)$.
Suppose for a moment that the following lower estimate is valid for each $i \in \{1,\dots,k\}$:
\begin{equation}\label{eq:Cheeger_inequality}
\lambda_1(p;E_i) \geq \left(\frac{h_1(E_i)}{p}\right)^p.
\end{equation}
Then for any $i \in \{1,\dots,k\}$ we get
$$
\mathfrak{L}_k(p;\Omega) 
\geq \lambda_1(p;E_i) \geq \left(\frac{h_1(E_i)}{p}\right)^p,
$$
and hence
$$
\mathfrak{L}_k(p;\Omega) \geq \left(\frac{1}{p}\max\limits_{i=1,\dots,k} h_1(E_i)\right)^p \geq 
\left(\frac{h_k(\Omega)}{p}\right)^p,
$$
where the second inequality is obtained by the definition \eqref{hkmax2}.
Letting now $p \to 1$, we get the desired lower bound \eqref{eq:limit2}. Combining it with \eqref{eq:limit1}, we conclude finally that 
$$
\lim_{p \to 1} \mathfrak{L}_k(p;\Omega) = h_k(\Omega).
$$

Let us now prove the estimate \eqref{eq:Cheeger_inequality}. 
Note that \eqref{eq:Cheeger_inequality} is valid for bounded open sets, see \cite[Appendix]{Lefton}. However, in general, our $E_i$'s are only measurable (or $p$-quasi-open by Lemma \ref{lem:pqo}).
We will detalize the proof of \cite{Lefton} in order to cover our case.
Let $v \in W_0^{1,p}(\Omega) \setminus \{0\}$ be a minimizer of $\lambda_1(p;E_i)$. Denoting $\Phi(v) := |v|^{p-1}v$, we get
\begin{equation}\label{eq:Phiv}
\int_\Omega |\nabla \Phi(v)| \, dx = p \int_\Omega |v|^{p-1} |\nabla v| \, dx \leq
p \left(\int_\Omega |v|^p \, dx \right)^\frac{p-1}{p}
\left(\int_\Omega |\nabla v|^p \, dx \right)^\frac{1}{p},
\end{equation}
which implies	 that $\Phi(v) \in W_0^{1,1}(\Omega)$.
Since $v \equiv 0$ a.e.\ on $\Omega \setminus E_i$ and we can assume that $v \geq 0$ a.e.\ on $\Omega$, we have $\Phi(v) \equiv 0$ a.e.\ on $\Omega \setminus E_i$ and $\Phi(v) \geq 0$ a.e.\ on $\Omega$. 
Therefore, denoting 
$$
F_t := \{x \in \Omega:~ \Phi(v(x)) > t\} \quad \text{for} \quad t \geq 0,
$$
and arguing as in the proof of \cite[Corollary 2.2.3]{kesavan}, we see that $P(F_t; \Omega) = P(F_t; \mathbb{R}^N) =: P(F_t)$ for all $t > 0$. Consequently, using the co-area formula (cf.\ \cite[Theorem 2.2.1 and Corollary 2.2.1]{kesavan}) with the layer-cake representation, and noting that $F_t \subset E_i$ for $t>0$, we derive 
\begin{align*}\label{eq:co-area1}
\int_{E_i} |\nabla \Phi(v)| \, dx 
&= \int_{\Omega} |\nabla \Phi(v)| \, dx 
= \int_{0}^{+\infty} P(F_t) \,dt 
= \int_{0}^{+\infty} \frac{P(F_t)}{|F_t|} \, |F_t| \,dt \\
&	\geq \inf_{t > 0} \frac{P(F_t)}{|F_t|} \int_{0}^{+\infty} |F_t| \,dt
\geq \inf_{D \subset E_i} \frac{P(D)}{|D|} \int_{\Omega} |\Phi(v)| \, dx 
= h_1(E_i) \int_{E_i} |v|^p \, dx.
\end{align*}
Finally, applying the inequality \eqref{eq:Phiv}, we get
$$
h_1(\Omega_i) \leq p \frac{\left(\int_{E_i} |\nabla v|^p \, dx \right)^\frac{1}{p}}{\left(\int_{E_i} |v|^p \, dx \right)^\frac{1}{p}} = p \left(\lambda_1(p;E_i)\right)^\frac{1}{p}, 
$$
and hence \eqref{eq:Cheeger_inequality} follows.
\end{proof}

\section{Applications}\label{sec:applications}

\subsection{Second Cheeger constant for a ring}\label{sec:ring}
Let $\Omega$ be a concentric ring in $\mathbb{R}^2$. We can assume, without loss of generality, that $\Omega = B_1 \setminus \overline{B_R}$, where $B_1$ and $B_R$ are open concentric discs with radii $1$ and $R$, respectively, and $R \in (0,1)$. 
We start with a discussion of a configuration for coupled Cheeger sets in $\Omega$ which is empirically optimal. We will call it \textit{reference configuration}. 
Then we prove its optimality rigorously. 

\subsubsection{Reference configuration}\label{sec:reference_configuration}
Let $\Omega' := \Omega \cap \{(x,y) \in \mathbb{R}^2: y > 0\}$ be the upper half-ring corresponding to $\Omega$, see Fig.~\ref{fig:Fig0}. 
Let us compute $h_1(\Omega')$. 
For this end, we consider a set $\mathcal{O}_r$ obtained by rolling a disc with fixed radius $r \in \left(0, \frac{1-R}{2}\right]$ inside of $\Omega'$, i.e.,
\begin{equation}\label{eq:Or}
\mathcal{O}_r = \bigcup_{x \in [\Omega']^r} B_r(x),
\end{equation}
where $[\Omega']^r$ is the inner parallel set to $\Omega'$ at distance $r$, that is,
$$
[\Omega']^r := \{x \in \Omega':~ \text{dist}\,(x, \partial \Omega') \geq r\}.
$$
Denote the quotient perimeter/area of $\mathcal{O}_r$ by $\mathcal{F}(r)$, i.e.,
\begin{equation*}\label{eq:ring_referenceFraction}
\mathcal{F}(r) =
\frac{|\text{Arc}(B_1 B_1')| + |\text{Arc}(B_2 B_2')| + 2|\text{Arc}(A_1 B_1)| + 2|\text{Arc}(A_2 B_2)| + 2|A_1 D|}{|\Omega'| - 2|A_1 B_1 C_1| - 2|A_2 B_2 C_2|},
\end{equation*}
where for the numerator we have
\begin{align*}
&|\text{Arc}(B_1 B_1')| = r \left( \pi - 2\arcsin\left(\frac{r}{1-r}\right) \right),
&|\text{Arc}(B_2 B_2')| = r \left( \pi - 2\arcsin\left(\frac{r}{R+r}\right) \right);\\
&|\text{Arc}(A_1 B_1)| = r \left( \frac{\pi}{2} + \arcsin\left(\frac{r}{1-r}\right) \right),
&|\text{Arc}(A_2 B_2)| = r \left( \frac{\pi}{2} - \arcsin\left(\frac{r}{R+r}\right) \right);\\
&|A_1 D| = |O A_1| - |O A_2| = \sqrt{1-2r} - \sqrt{R(2r+R)},
\end{align*}
and for the denominator we have
\begin{align*}
|\Omega'| & = \frac{\pi(1-R^2)}{2},\\
|A_1 B_1 C_1| 
&= |C_1 O B_1| - |A_1 O_1 B_1| - |A_1 O O_1| \\
&= \frac{1}{2}\arcsin\left(\frac{r}{1-r}\right) - \frac{r^2}{2} \left(\frac{\pi}{2} + \arcsin\left(\frac{r}{1-r}\right)\right) - \frac{r}{2}\sqrt{1-2r}
\end{align*}
and
\begin{align*}
|A_2 B_2 C_2| 
&= |O_2 O A_2| - |B_2 O C_2| - |A_2 O_2 B_2| \\
&= \frac{r}{2}\sqrt{R(2r+R)} - \frac{R^2}{2}\arcsin\left(\frac{r}{R+r}\right) - \frac{r^2}{2} \left(\frac{\pi}{2} - \arcsin\left(\frac{r}{R+r}\right)\right).
\end{align*}

\begin{lem}\label{lem:h1(Omega')}
	$h_1(\Omega') = \min\limits_{r \in \left[0, \frac{1-R}{2}\right]} \mathcal{F}(r)$. Moreover, a minimizer of $h_1(\Omega')$ is unique and given by $\mathcal{O}_{r_0}$ with some $r_0 \in \left(0, \frac{1-R}{2}\right)$.
\end{lem}
\begin{proof}
	Evidently, each $\mathcal{O}_r$ with $r \in \left[0, \frac{1-R}{2}\right]$ is an admissible set for the minimization problem $h_1(\Omega')$, that is,
	\begin{equation*}\label{eq:h1_upper_bound}
	h_1(\Omega') \leq \mathcal{F}(r_0) := \min_{r \in \left[0, \frac{1-R}{2}\right]} \mathcal{F}(r).
	\end{equation*}
	To show that $h_1(\Omega') = \mathcal{F}(r_0)$, let us recall the following definition from \cite{saracco}. 
	It is said that $\Omega'$ \textit{has no necks of radius} $r$ if for any two balls $B_r(x_0), B_r(x_1) \subset \Omega'$ there exists a continuous curve $\gamma: [0, 1] \to \Omega'$ such that
	$$
	\gamma(0) = x_0,
	\quad 
	\gamma(1) = x_1, 
	\quad \text{and} \quad 
	B_r(\gamma(t)) \subset \Omega' 
	\quad \text{for all} \quad  t \in [0, 1].
	$$
	It is not hard to observe that $\Omega'$ satisfies this property for all $r \in \left[0, \frac{1-R}{2}\right]$. Moreover, this property is also satisfied for all $r > \frac{1-R}{2}$ since $\Omega'$ contains no ball of such radius $r$. 
	Applying now \cite[Theorem 1.4]{saracco}, we conclude that the Cheeger set of $\Omega'$ is given by $\mathcal{O}_{r_0}$, that is, $h_1(\Omega') = \mathcal{F}(r_0)$. Moreover, $\mathcal{O}_{r_0}$ is a unique minimizer of $h_1(\Omega')$. 
	
	Let us show that $r_0 \in \left(0, \frac{1-R}{2}\right)$. The case $r_0 = 0$ is impossible, since otherwise we get a contradiction to \cite[Theorem 1.4]{saracco}.
	On the other hand, if $r_0 = \frac{1-R}{2}$, then we must have $h_1(\Omega') = \frac{2}{1-R}$, cf.\ \cite[Proposition 3.5 (iv)]{leonardi} or Proposition \ref{prop:meancurvature1}. However, we see that 
	\begin{equation*}\label{eq:ring_contr0}
	h_1(\Omega') = \mathcal{F}\left(\frac{1-R}{2}\right) = 
	\frac{2}{1-R} \cdot \frac{4\pi - 4(1+R) \arcsin\left(\frac{1-R}{1+R}\right)}{\pi (3+R) - 4 (1+R) \arcsin\left(\frac{1-R}{1+R}\right)} > \frac{2}{1-R} = h_1(\Omega'),
	\end{equation*}
	which is impossible.
\end{proof}

\begin{figure}[ht]
	\begin{minipage}[t]{0.49\linewidth}
		\centering
		\includegraphics[width=1\linewidth]{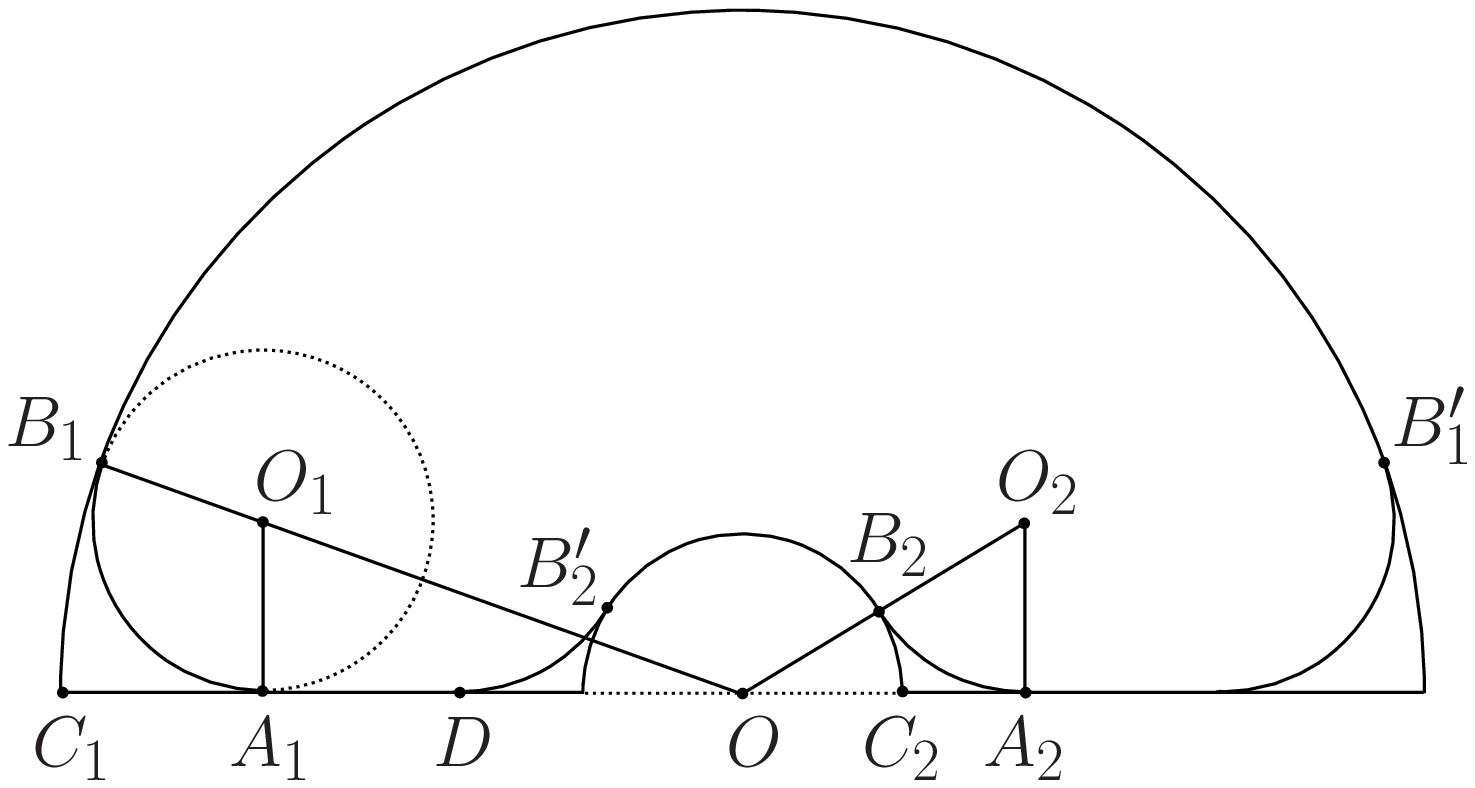}
		\caption{}
		\label{fig:Fig0}
	\end{minipage}
	\hfill
	\begin{minipage}[t]{0.49\linewidth}
		\centering
		\includegraphics[width=0.96\linewidth]{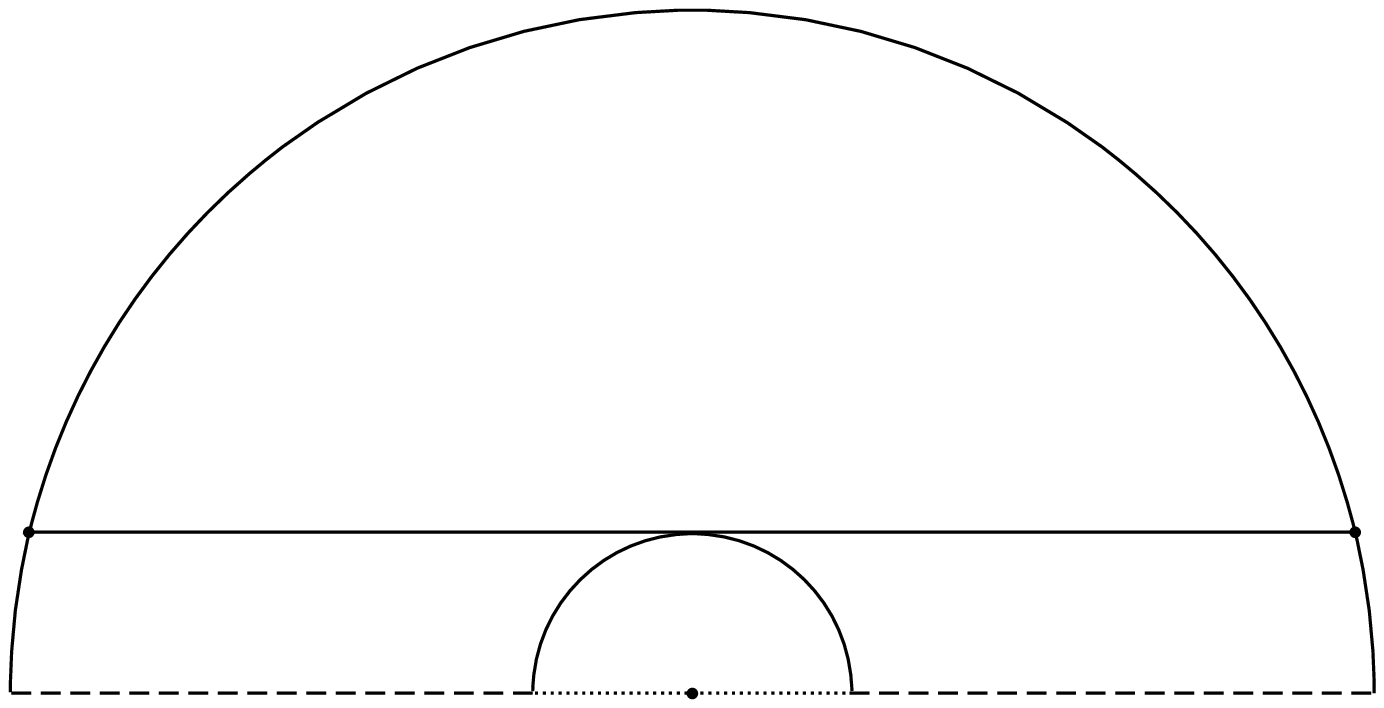}
		\caption{}
		\label{fig:Fig4}
	\end{minipage}
\end{figure}

\subsubsection{Optimality}\label{sec:reference_configuration_optimal}
Let us show that the reference configuration is indeed a minimizer of $h_2(\Omega)$ and it is unique up to rotation. 
Denote by $(E_1, E_2) \in \mathcal{E}_2$ any $1$-adjusted Cheeger couple of $\Omega$, that is, $(E_1, E_2)$ is a minimizer of $h_2(\Omega)$ and 
$$
h_1(\Omega \setminus E_1) = h_1(E_2) = \frac{P(E_2)}{|E_2|}, 
\quad 
h_1(\Omega \setminus E_2) = h_1(E_1) = \frac{P(E_1)}{|E_1|}.
$$
Note that $(E_1, E_2)$ exists by Theorem~\ref{thm:main1}. Moreover, evidently, $(E_1, E_2)$ is $2$-adjusted.
We assume that each $E_i$ has only one connected component.
Throughout the proof we will frequently use the following properties of $E_1$ and $E_2$ which follow from Theorem \ref{thm:interiorregularity} and Propositions \ref{prop:meancurvature1} and \ref{prop:meancurvature2}:
\begin{itemize}
	\item $\partial E_i \subset \overline{\Omega}$ is $C^1$-smooth.
	\item The mean curvatures of the free boundaries $\partial_f E_1$ and $\partial_f E_2$ are constant. We will denote them as $c_1$ and $c_2$, respectively.
	\item The mean curvature of the contact surface $\partial(E_1 E_2)$ is constant. We will denote it as $c_{12}$.
	\item $\partial E_i$ consists of arcs of circles and $\partial E_i = \partial_f E_i	\cup \partial_b E_i \cup \partial (E_1E_2)$.
	\item If $\partial (E_1E_2)$ is not closed, then each end-point of $\partial (E_1E_2)$ lies inside $\Omega$ and
	gives rise to exactly two arcs of free boundaries $\partial_f E_i$. 
	\item Each end-point of $\partial_f E_i$ lies either on $\partial B_1$ or on $\partial B_R$ or coincides with an end-point of $\partial (E_1E_2)$.
\end{itemize}
We will say that a part of $\partial E_i$ has \textit{nonempty interior} if this part is not empty and not discrete.

The proof will be performed through the following steps:
\begin{enumerate}[label={\rm(\roman*)}]
	\item\label{proof:annuli:i}		Show that $\partial(E_1 E_2)$ has nonempty interior. 
	\item\label{proof:annuli:ii}	Show that $\partial E_1 \cap \partial B_R$ or $\partial E_2 \cap \partial B_R$ has nonempty interior. 
	\item\label{proof:annuli:iii}	Show that $\partial(E_1 E_2)$ does not have closed connected components. 
	\item\label{proof:annuli:iv}	Show that $c_1 = c_2$. 
	\item\label{proof:annuli:v}		Show that $c_{12} = 0$.
	\item\label{proof:annuli:vi}	Show the optimality of the reference configuration. 
\end{enumerate}

\ref{proof:annuli:i} We start with proving that the contact surface between $E_1$ and $E_2$ has nonempty interior. 
Suppose the claim was false. Thus, we have
\begin{equation}\label{eq:e1e2}
\partial E_1 = \overline{\partial_f E_1}
\cup \partial_b E_1
\quad \text{and} \quad 
\partial E_2 = \overline{\partial_f E_2}
\cup \partial_b E_2,
\end{equation}
where the closure is taken with respect to the relative topology. Note that $\partial_f E_1$ and $\partial_f E_2$ have nonempty interiors. Indeed, if we suppose that, say, the interior of $\partial_f E_1$ is empty, then $\partial E_1 = \partial_b E_1$. This readily yields $\partial E_1 = \partial \Omega$ and $E_1 = \Omega$, and hence $E_2 = \emptyset$, which is impossible. By the same reasoning, $\partial_f E_2$ also has nonempty interior.

Let us show that $\partial_f E_1$ does not have closed connected components. 
Suppose, by contradiction, that there is a closed connected component $S$ of $\partial_f E_1$. Then, recalling that the curvature of $\partial_f E_1$ is a positive constant $c_1$, we see that $S$ is a circle. 
Denoting by $B_0$ a ball such that $\partial B_0 = S$, we obtain $E_1 \subset B_0$. 
There are two possible positions for $B_0$: either $B_0 \subset \Omega$ or $B_R \subset B_0$. 
If $B_0 \subset \Omega$, then we get a contradiction since such a configuration is not optimal. Indeed, a maximal ball which can be inscribed in $\Omega$ has radius $\frac{1-R}{2}$. Moreover, $\Omega$ contains at least two such balls. Therefore, it is not hard to see that in the best configuration with $E_1 \subset B_0$ it will hold $E_1 = B_0$ and $\text{radius}(B_0) = \frac{1-R}{2}$, that is,
$$
h_2(\Omega) = \max\{h_1(B_0), h_1(\Omega \setminus B_0)\} = h_1(B_0) = \frac{2}{\text{radius}(B_0)} = \frac{4}{1-R}.
$$
However, taking two half-rings as a pair of admissible sets for $h_2(\Omega)$, we easily obtain a contradiction:
\begin{equation}\label{eq:ring_contr1}
h_2(\Omega) \leq 2\, \frac{\pi(1+R) + 2(1-R)}{\pi(1-R^2)} < \frac{4}{1-R} = h_2(\Omega).
\end{equation}
Suppose now that $B_R \subset B_0$. 
Then it is not hard to see that $E_2 \subset B_1 \setminus B_0$. 
Therefore, in view of \eqref{eq:e1e2}, we see that $E_2$ must be a ball, and the contradiction follows as in \eqref{eq:ring_contr1}.
Thus, $\partial_f E_1$ does not have closed connected components. 
Analogously, the same conclusion holds for $\partial_f E_2$.

Fix now any connected component of $\partial_f E_1$.
Since this component is not closed and it is an arc of a circle of radius $\frac{1}{c_1}$, and $\partial(E_1 E_2)$ has empty interior by the assumption, we see that this component must touch both $\partial B_1$ and $\partial B_R$. Thus, either $\frac{1}{c_1} = \frac{1-R}{2}$ or $\frac{1}{c_1} = \frac{1+R}{2}$. However, the latter case is impossible due to the $C^1$-smoothness of $\partial E_1$.
Therefore, we conclude that $h_1(E_1) = c_1 = \frac{2}{1-R}$. 
Let us show that in this case the best configuration for $E_1$ and $E_2$ will be given by $\mathcal{O}_{\frac{1-R}{2}}$, see \eqref{eq:Or}. 
For this end, we study the behaviour of $\partial_f E_1$ and $\partial_f E_2$, see Fig.~\ref{fig:Fig2}.
Let an arc $PQ \subset \partial_f E_1$ be such that $P \in \partial B_1$ and $Q \in \partial B_R$. 
If there exists an arc $PP_1 \subset \partial_f E_1$ such that $P_1 \in \Omega \setminus PQ$, or an arc $QQ_1 \subset \partial_f E_1$ such that $Q_1 \in \Omega \setminus PQ$, then, recalling that $\partial(E_1E_2)$ has empty interior, this arcs can be prolonged such that $P_1=Q$ or $Q_1 = P$, and we obtain a contradiction as in \eqref{eq:ring_contr1}.
Hence, there are arcs $PR \subset \partial E_1 \cap \partial B_1$ and $QS \subset \partial E_1 \cap \partial B_R$. 
Moreover, noting that the same holds true for $\partial E_2$, we conclude that $PR \not\equiv \partial B_1$ and $QS \not\equiv \partial B_R$, and we can take $R$ and $S$ such that $RS \subset \partial_f E_1$. Thus, we see that $E_1$ and $E_2$ have the shapes as depicted in Fig.~\ref{fig:Fig2}, and at least one of $E_i$ is contained in a half-ring of $\Omega$. It is not hard to show that the best configuration for each $E_i$ will be given, up to rotation, by $\mathcal{O}_{\frac{1-R}{2}}$ (see \eqref{eq:Or}). However, $\mathcal{O}_{\frac{1-R}{2}}$ is not optimal, as it follows from Lemma \ref{lem:h1(Omega')}. 
Therefore, we conclude that $\partial(E_1 E_2)$ has nonempty interior.

\ref{proof:annuli:ii} Assume, without loss of generality, that $c_{12} \geq 0$. 
Let us show that $\partial E_1 \cap \partial B_R$ has nonempty interior. 
Suppose, by contradiction, that the interior of $\partial E_1 \cap \partial B_R$ is empty. 
Since the curvature of $\partial_f E_1$ is $c_1 > 0$ and the curvature of $\partial E_1 \cap \partial B_1$ equals $1$ (whenever these parts have nonempty interiors), and the curvature of $\partial (E_1 E_2)$ is $c_{12} \geq 0$, we conclude that $E_1$ has a piecewise nonnegative curvature. Combining this fact with the $C^1$-smoothness of $\partial E_1$, we conclude that $E_1$ is convex. 
Evidently, the largest convex set which can be inscribed in $\Omega$ is a circular segment as on Fig.~\ref{fig:Fig4}. 
Therefore, $E_1$ must be contained inside such segment and, consequently, strictly inside a half-ring. But this fact contradicts Lemma~\ref{lem:h1(Omega')}. 

Note that if $c_{12} = 0$, then the above arguments can be applied to both $E_1$ and $E_2$ to conclude that $\partial E_1 \cap \partial B_R$ and $\partial E_2 \cap \partial B_R$ have nonempty interiors. 

\ref{proof:annuli:iii} Let us now prove that the contact surface between $E_1$ and $E_2$ does not have closed connected components. 
Suppose the assertion is false. 
Since the curvature of $\partial(E_1 E_2)$ is a constant $c_{12}$, there exists a connected component $S$ of $\partial(E_1 E_2)$ which is a circle. 
Let $B_0$ be a ball such that $\partial B_0 = S$. 
We again have two possibilities: either $B_0 \subset \Omega$ or $B_R \subset B_0$. 
If $B_0 \subset \Omega$, then we get a contradiction as in \eqref{eq:ring_contr1}. 
Suppose now that $B_R \subset B_0$. 
Since each $E_i$ has only one connected component, we can assume, without loss of generality, that $E_1 \subset B_0 \setminus \overline{B_R}$ and $E_2 \subset B_1 \setminus \overline{B_0}$.  
Let us show that, in fact,  $E_1 = B_0 \setminus \overline{B_R}$ and $E_2 = B_1 \setminus \overline{B_0}$.
Suppose, by contradiction, that, say, $E_2$ has a free boundary. Since $\partial B_0 \subset \partial (E_1E_2)$ and $\partial E_2$ is $C^1$-smooth, we see that there is no arc of $\partial_f E_2$ with an end-point on $\partial B_0$. Hence, the only possibility is that $\partial_f E_2$ is a circle $\partial B_s$ such that $\partial B_s \cap \Omega$ is of angle $2\pi$, which contradicts \cite[Lemma~2.11]{leonardipratelli}.
Similarly, we can conclude that $E_1 = B_0 \setminus \overline{B_R}$. 
Varying the radius of $B_0$, it is easy to see that the best configuration is achieved when $h_1(E_1) = h_1(E_2)$. Denoting the radius of $B_0$ as $s$ and noting that 
$$
h_1(E_1) = \frac{P(B_0 \setminus \overline{B_R})}{|B_0 \setminus \overline{B_R}|} = \frac{2}{s-R},
\quad 
h_1(E_2) = \frac{P(B_1 \setminus \overline{B_0})}{|B_1 \setminus \overline{B_0}|} = \frac{2}{1-s},
$$
we obtain that $s = \frac{1+R}{2}$, and hence $h_2(\Omega) = \frac{4}{1-R}$. The contradiction then follows as in \eqref{eq:ring_contr1}.
Thus, $\partial(E_1 E_2)$ does not have closed connected components. 

\ref{proof:annuli:iv} As a consequence of step \ref{proof:annuli:iii}, we see that the free boundaries of $E_1$ and $E_2$ are not empty. 
Let us prove that $c_1 = c_2$. 
Suppose, without loss of generality, that $c_{12} \geq 0$. 
Then we get from step \ref{proof:annuli:ii} that $\partial E_1 \cap \partial B_R$ has nonempty interior. 
We know from Proposition \ref{prop1} that $c_1 \geq c_2$ and $c_{1} \geq c_{12}$, and if $c_1 > c_2$, then $c_1=c_{12}$. We prove that $c_1 > c_{12}$, which will imply that the case $c_1 > c_2$ is impossible, and hence $c_1=c_2$. Suppose, contrary to our claim, that  $c_1 = c_{12} \geq c_2$. If $\partial B_R \not\subset \partial E_1$, then we argue as in step \ref{proof:annuli:i} to deduce that the best configuration for $E_1$ and $E_2$ in this case is given by $\mathcal{O}_{\frac{1-R}{2}}$, which is impossible. Therefore, $\partial B_R \subset \partial E_1$. 
Then, the $C^1$-smoothness of $\partial E_1$ implies that no arc of $\partial_f E_1$ can have an end-point on $\partial B_R$. Hence, since $c_1 = c_{12}$, we conclude that there exists a circle $\partial B_0 \subset \partial E_1$. Suppose that $B_0 \subset \Omega$. Since $E_1$ has only one connected component and $c_1=c_{12} > 0$, we see that $E_1 \subset B_0$. However, it contradicts the fact that $\partial B_0 \subset \partial E_1$. 
Suppose now that $B_R \subset B_0$. Since, again, $E_1$ has only one connected component and $c_1=c_{12} > 0$, we deduce that $E_1 \subset B_0 \setminus \overline{B_R}$ and $E_2 \subset B_1 \setminus \overline{B_0}$. Moreover, it is easy to see that $E_1 = B_0 \setminus \overline{B_R}$. Denoting by $s$ the radius of $B_0$ and recalling that $\partial_f E_1 \cap \partial B_0$ has nonempty interior, we get the following contradiction:
$$
h_1(E_1) = \frac{2}{s-R} = \frac{1}{s} = c_1
\quad \implies \quad 
s < 0.
$$
Therefore, $c_1 \neq c_{12}$. Finally, applying Proposition~\ref{prop1}, we conclude that $c_1 = c_2$. 

\ref{proof:annuli:v} Let us show now that $c_{12} = 0$.
Suppose, by contradiction and without loss of generality, that $c_{12} > 0$. 
To prove the claim, we will study the behaviour of $\partial E_1$ and $\partial E_2$. 

\begin{figure}[ht]
	\begin{minipage}[t]{0.48\linewidth}
		\centering
		\includegraphics[width=1\linewidth]{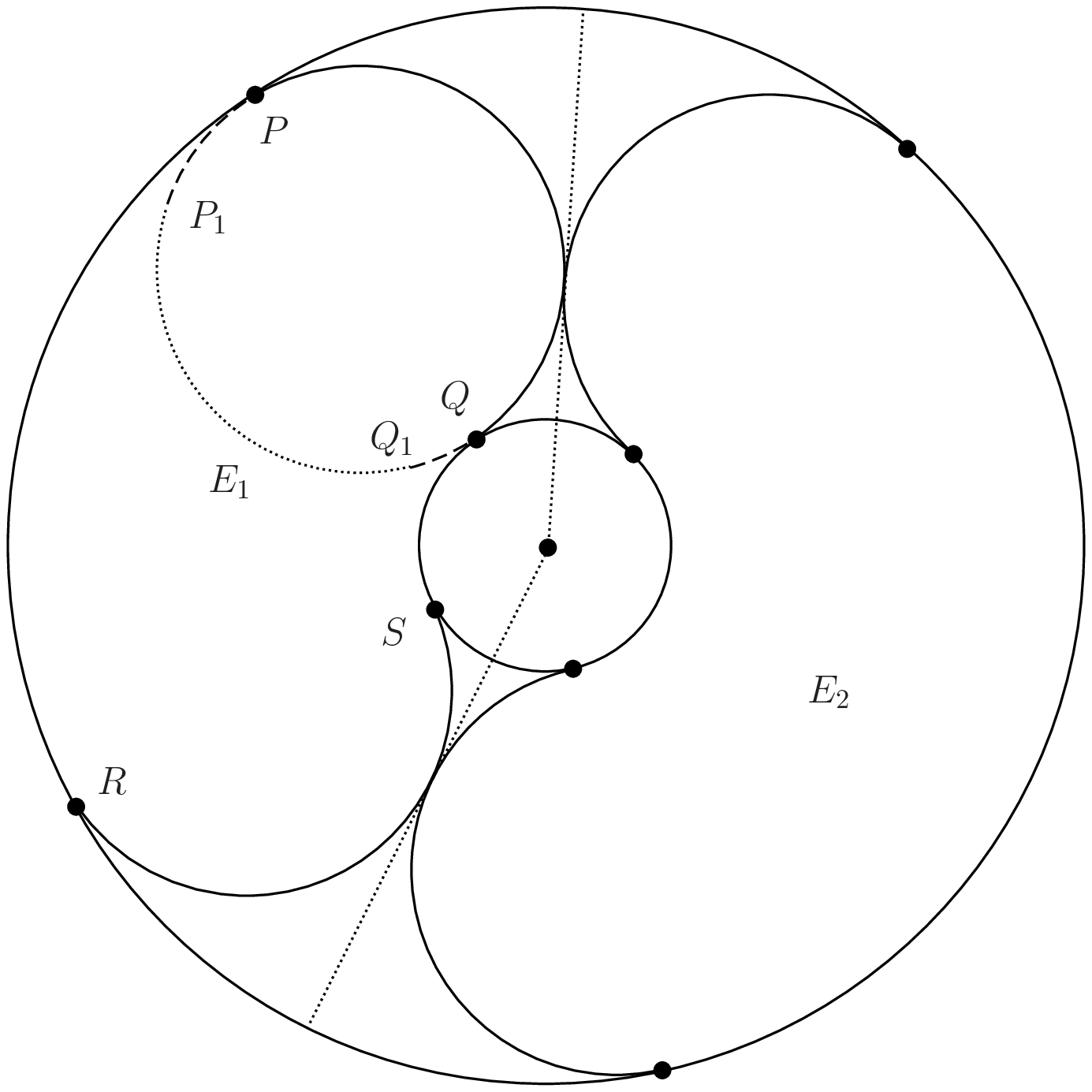}
		\caption{}
		\label{fig:Fig2}
	\end{minipage}
	\hfill
	\begin{minipage}[t]{0.48\linewidth}
		\centering
		\includegraphics[width=1\linewidth]{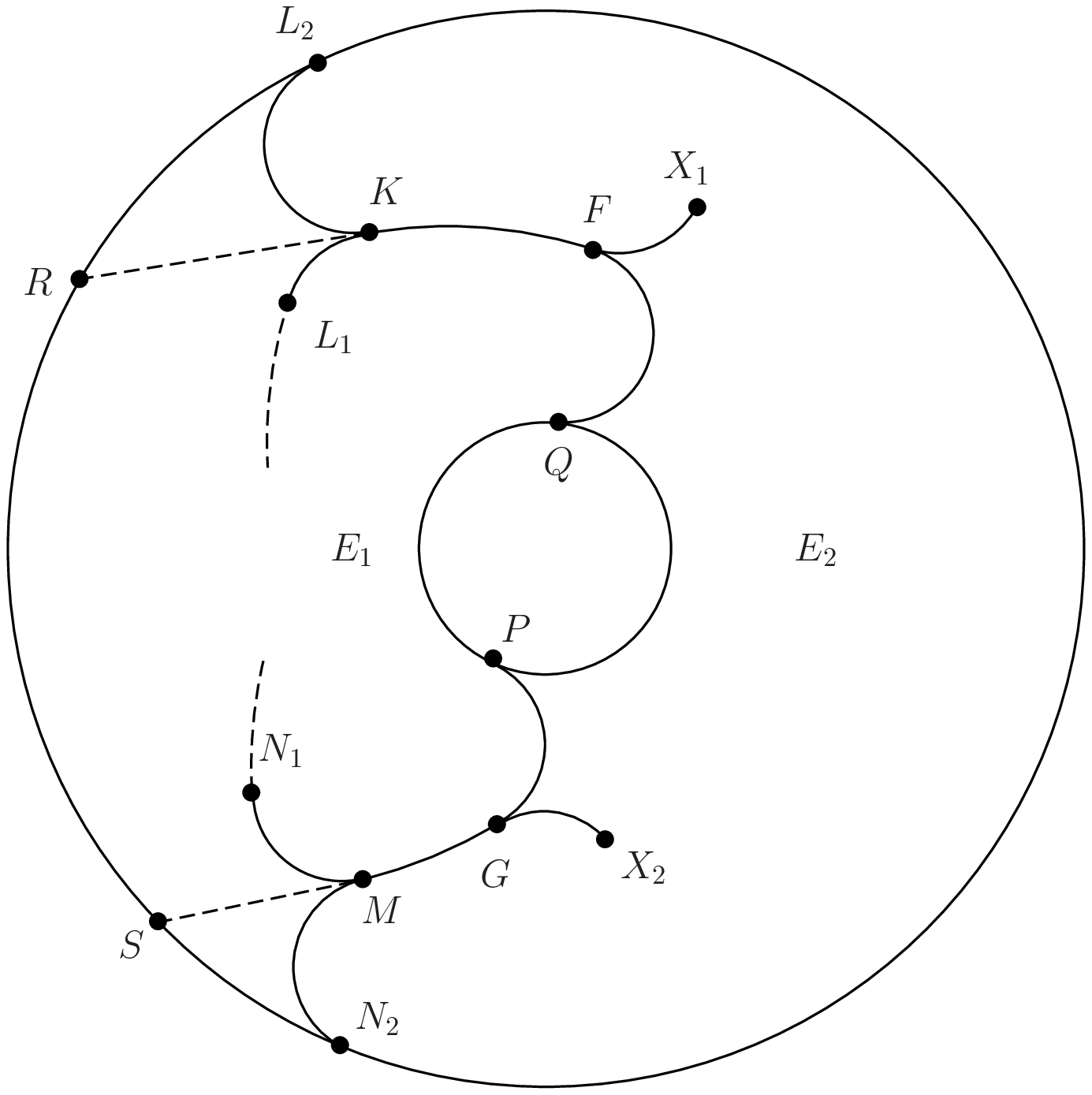}
		\caption{}
		\label{fig:Figx5}
	\end{minipage}
\end{figure}

From step \ref{proof:annuli:ii} we know that $\partial E_1 \cap \partial B_R$ has nonempty interior. 
Take an arc $PQ \subset \partial E_1 \cap \partial B_R$ and points $F, G \in \partial(E_1 E_2)$ such that the arc $QF \subset \partial_f E_1$ and $PG \subset \partial_f E_1$, see Fig.~\ref{fig:Figx5}. 
Note that such $F$ and $G$ exist, since otherwise, without loss of generality, either $F \in \partial B_R$ or $F \in \partial B_1$. The first case is impossible, since then $F=Q$ and the contradiction follows, for instance, as in \eqref{eq:ring_contr1}. The second case is impossible, since then the best configuration will be given by $\mathcal{O}_{\frac{1-R}{2}}$ which is not optimal, see step \ref{proof:annuli:i} and Lemma \ref{lem:h1(Omega')}.
Consider maximal arcs $FK, GM \subset \partial (E_1 E_2)$. 

1) Suppose first that $FK \equiv GM$, i.e., $F = M$ and $K = G$, see Fig.~\ref{fig:Figx6}.
Consider arcs $FX_1, GX_2 \subset \partial_f E_2$. 
There are several possible positions of $X_1$ and $X_2$: 

\begin{itemize}
	\item[(a)] Assume first that $X_1, X_2 \in \partial B_1$. 
	Let $GF$ be the arc of a circle $\partial B_0$. Let us show that $B_0$ is concentric with $B_1$ and $B_R$. Suppose the claim is false. Noting that the arcs $PG$ and $GX_2$ cannot be of angle greater than $\pi$ \cite[Lemma~2.11]{leonardipratelli}, we easily get a contradiction, see Fig.~\ref{fig:Figx6}. 
	Therefore, $B_0$ is concentric with $B_1$ and $B_R$, and $E_1 \subset B_0 \setminus \overline{B_R}$ and $E_2 \subset B_1 \setminus \overline{B_0}$. However, it is known that rings $B_0 \setminus \overline{B_R}$ and $B_1 \setminus \overline{B_0}$ are calibrable \cite{Bueno}, that is, they are Cheeger sets of themselves. This implies that $E_1 = B_0 \setminus \overline{B_R}$ and hence $\partial(E_1E_2) = \partial B_0$, i.e., $\partial(E_1E_2)$ is closed. A contradiction to step \ref{proof:annuli:iii}.
	\item[(b)] Assume now that $X_1 = G$ and $X_2 = F$. That is, the arc $G X_2 \subset \partial_f E_2$ touches $PG$ and $QF$ at corresponding points.
	Since the arc $FG \subset \partial(E_1 E_2)$ also touches $PG$ and $QF$ at the same points as $GX_2$, we readily get a contradiction to the fact that $c_2, c_{12} > 0$.
	\item[(c)] Other positions for $X_1$ and $X_2$ are impossible by evident reasons. 
\end{itemize}

\begin{figure}[ht]
	\begin{minipage}[t]{0.48\linewidth}
		\centering
		\includegraphics[width=1\linewidth]{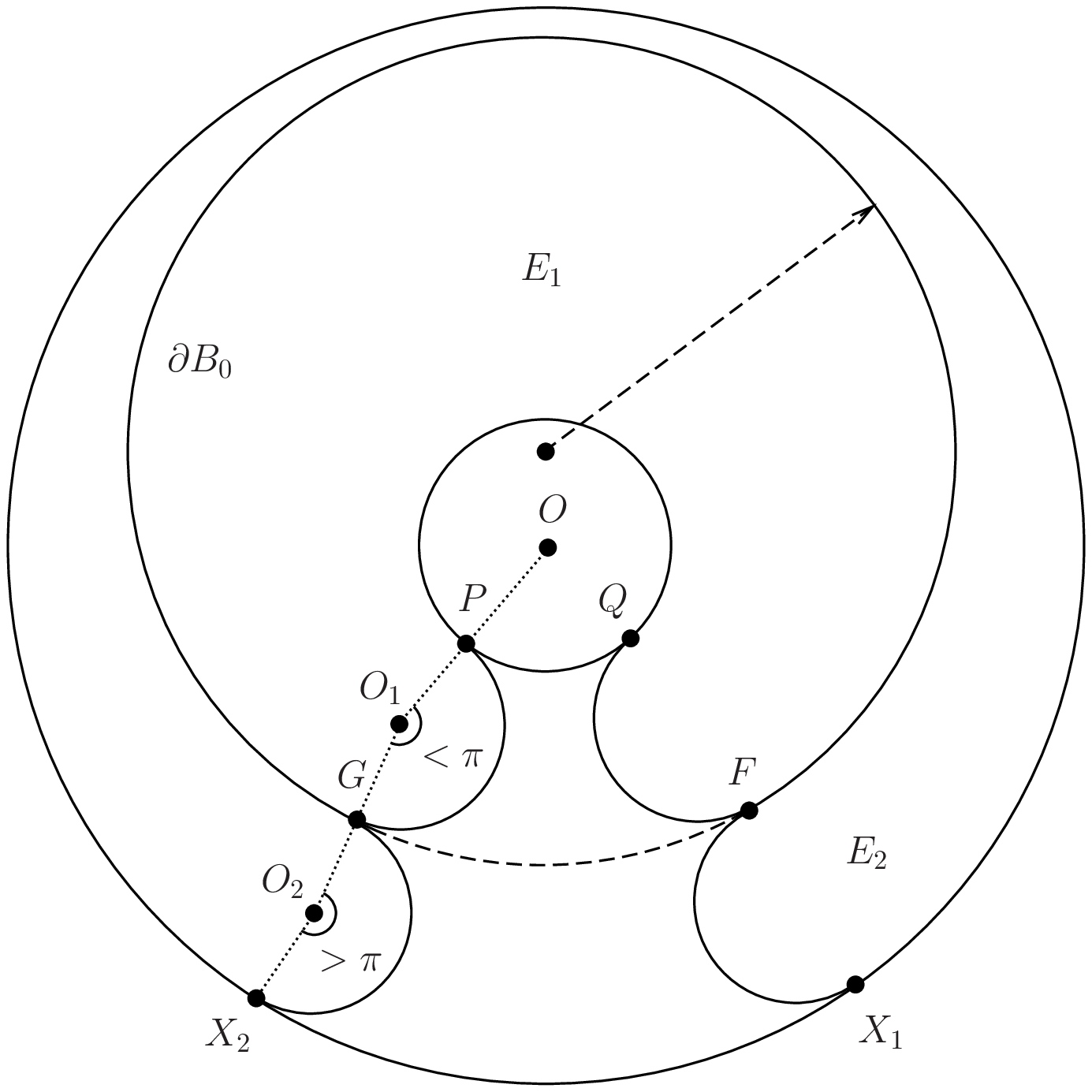}
		\caption{}
		\label{fig:Figx6}
	\end{minipage}
	\hfill
	\begin{minipage}[t]{0.48\linewidth}
		\centering
		\includegraphics[width=1\linewidth]{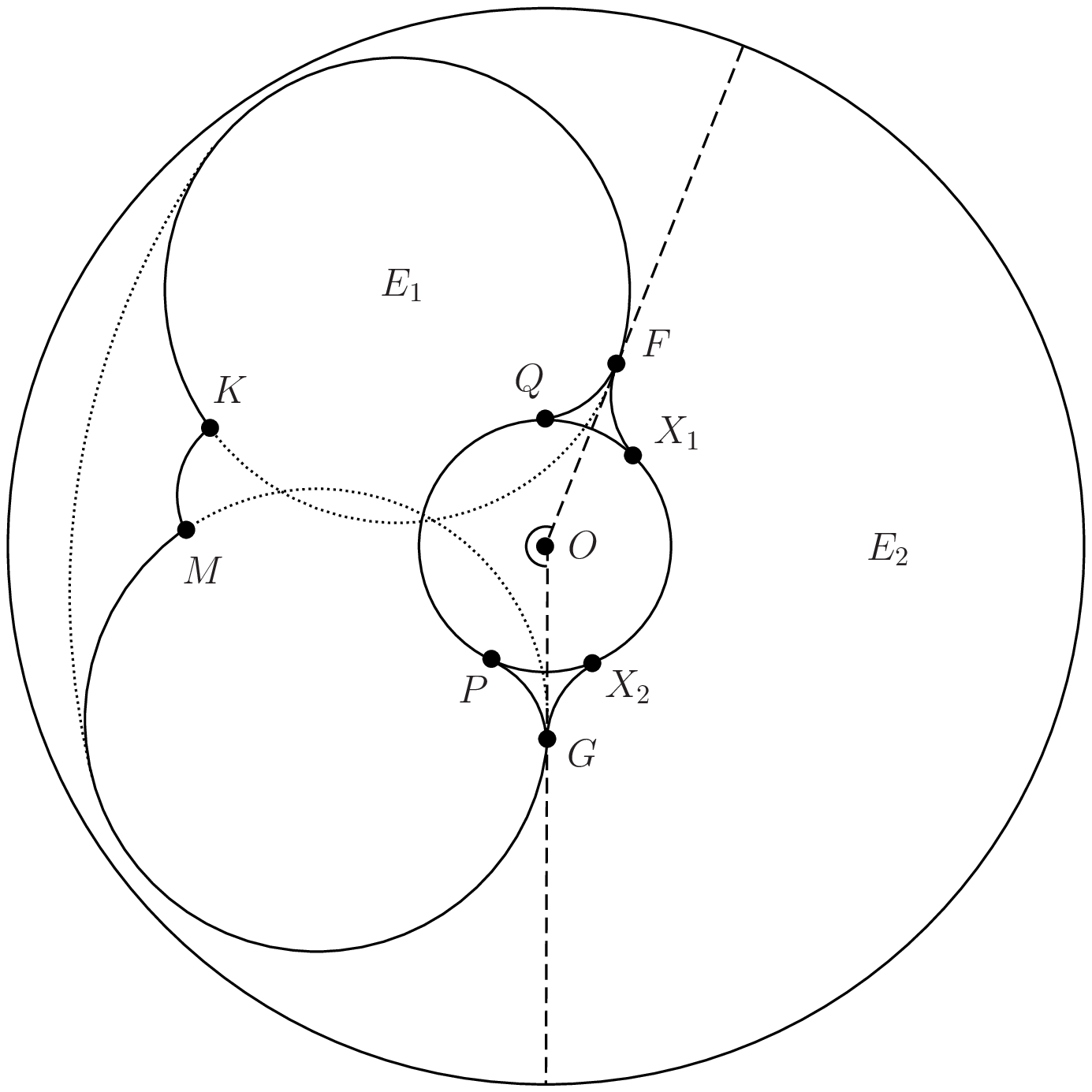}
		\caption{}
		\label{fig:Figx8}
	\end{minipage}
\end{figure}

2) Suppose now that $FK \not\equiv GM$, i.e., $F \neq M$ and $K \neq G$, see Fig.~\ref{fig:Figx5}.
Then, consider arcs $KL_1, MN_1 \subset \partial_f E_1$ and $KL_2, MN_2 \subset \partial_f E_2$. 
Consider also a segment $KR$ tangent to the arc $FK$ at the point $K$ such that $R \in \partial B_1$, and, analogously, a segment $MS$ tangent to the arc $GM$ at $M$ such that $S \in \partial B_1$. 
Recall that the curvature of $\partial E_1$ is piecewise nonnegative except of the part $\partial E_1 \cap \partial B_R$. 
Therefore, it is not hard to see that if $KR$ and $GM$ intersect at a point $T$ inside $\Omega$, then $E_1$ is contained inside a set bounded by a closed path 
\begin{equation}\label{eq:Z1_path1}
P \to \text{by the arc } \partial E_1 \cap \partial B_R \to Q \to F \to K \to T \to M \to G \to P.
\end{equation}
Analogously, if $KR$ and $GM$ do not intersect inside $\Omega$ (as it is depicted on Fig.~\ref{fig:Figx5}), then $E_1$ is contained inside a set bounded by a closed path 
\begin{equation}\label{eq:Z1_path2}
P \to \text{by the arc } \partial E_1 \cap \partial B_R \to Q \to F \to K \to R \to S \to M \to G \to P.
\end{equation}
Let us denote a set bounded by paths \eqref{eq:Z1_path1} or \eqref{eq:Z1_path2} as $\mathcal{Z}_1$. That is, $E_1 \subset \mathcal{Z}_1$.
This implies that $L_2, N_2 \in \partial B_1$. Indeed, if, say, $L_2 \in \partial(E_1 E_2)$, then the only possibility is $L_2 = M$. However, recalling that $c_{12} > 0$, it is not hard to observe that the arcs $FX_1, GX_2 \subset \partial_f E_2$ will intersect $\partial B_R$ transversally (see, for example, Fig.~\ref{fig:Figx1}), which is impossible in view of the $C^1$-smoothness of $\partial E_2$. The case $L_2 \in \partial B_R$ is obviously impossible, too. Thus, we conclude that $L_2 \in \partial B_1$. By the same arguments, $N_2 \in \partial B_1$. 

These facts readily imply that $E_2$ is located inside a set $\mathcal{Z}_2$ bounded by a closed path 
$$
Q \to F \to K \to L_2 \to  \text{by } \partial B_1 \setminus RS \to N_2 \to M \to G \to P \to \text{by } \partial B_R \setminus \partial E_1 \to Q.
$$

3) Suppose that $L_1 = M$. Consider arcs $FX_1, GX_2 \subset \partial_f E_2$. 
There are two cases for $X_1$:

\begin{enumerate}[label=(\alph*)]
	\item Assume first that $X_1 \neq G$ (and hence $X_2 \neq F$). Therefore, since $E_1 \subset \mathcal{Z}_1$, we see that $X_1, X_2 \not\in \partial (E_1E_2)$. 
	If $X_1 \in \partial B_1$, then we get at least two connected components of $E_2$, which is impossible. Analogously, $X_2 \not\in \partial B_1$. Therefore, we must have $X_1, X_2 \in \partial B_R$, see Fig.~\ref{fig:Figx8}.
	If the angle $\angle FOG$ measured from $E_1$ is less than or equal to $\pi$, then $E_1$ is contained strictly inside a half-ring, which contradicts Lemma~\ref{lem:h1(Omega')}.
	Thus, $\angle FOG > \pi$. 
	However, it is not hard to see that, in this case, $FK$ and $GM$ cannot be connected by $KL_1 = KM \subset \partial_f E_1$ in such a way that $\partial E_1$ is $C^1$-smooth, since $c_1 \geq c_{12}$. A contradiction. 
	\item Assume now that $X_1 = G$ (and hence $X_2 = F$). Then we see that $\angle FOG > \pi$ and, as above, we see that this case is impossible. 
\end{enumerate}

4) Suppose now that $L_1 \neq M$. Then we deduce that $L_1 \in \partial B_1$ and $N_1 \in \partial B_1$. Indeed, since $L_1 \neq M$, we have $L_1 \not\in \mathcal{Z}_2$ and hence $L_1 \not\in \partial(E_1 E_2)$. If $L_1 \in \partial B_R$, then $L_1 = P$ and $FK = GM$, which is impossible as shown in case 1) above. 
Thus, $L_1 \in \partial B_1$ and, by the same arguments, $N_1 \in \partial B_1$. 

5) At the end, we have $L_1, L_2, N_1, N_2 \in \partial B_1$. However, recalling that $c_{12} > 0$ by the assumption, and $c_1 = c_2$ by step \ref{proof:annuli:iv}, it is not hard to see that an arc $FX_1 \subset \partial_f E_2$ will intersect with $\partial B_R$ not transversely, see Fig.~\ref{fig:Figx1}. A contradiction to the $C^1$-smoothness of $\partial E_2$. 
Thus, $c_{12} = 0$.

\begin{figure}[ht]
	\begin{minipage}[t]{0.49\linewidth}
		\centering
		\includegraphics[width=1\linewidth]{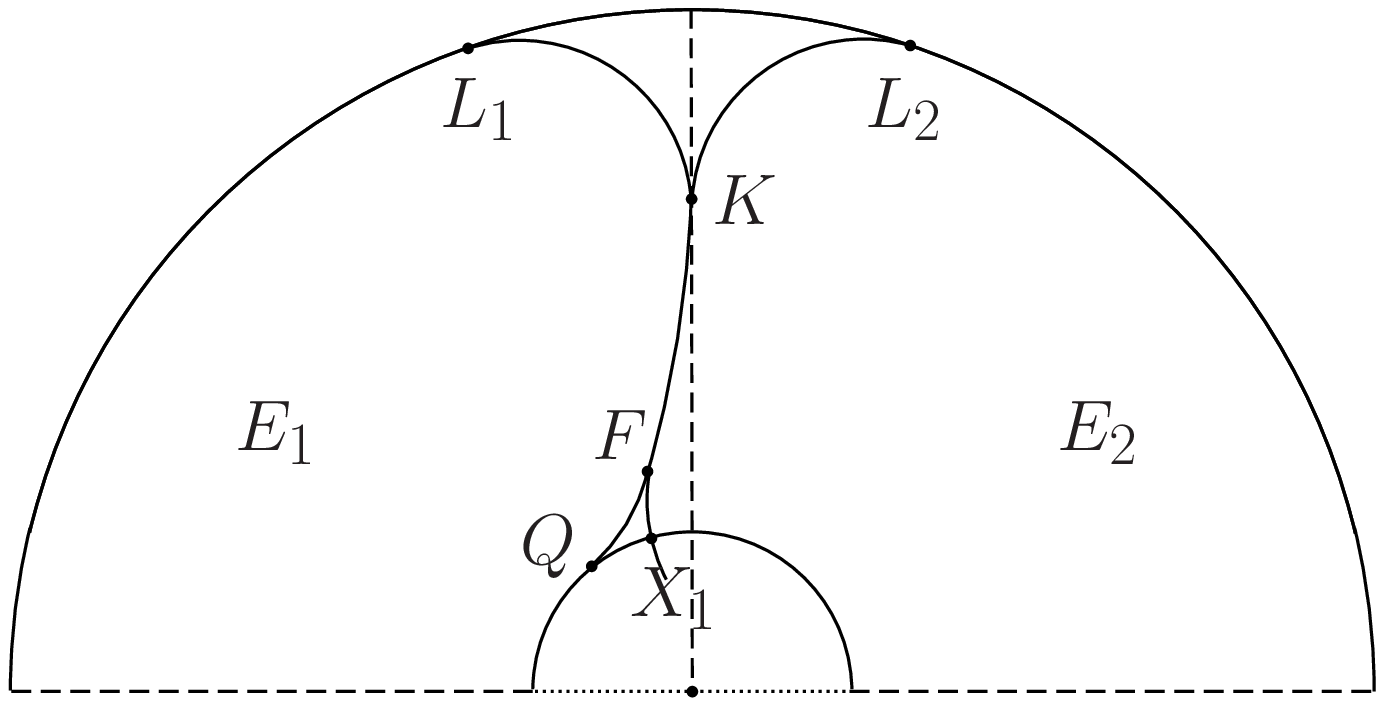}
		\caption{}
		\label{fig:Figx1}
	\end{minipage}
	\hfill
	\begin{minipage}[t]{0.49\linewidth}
		\centering
		\includegraphics[width=1\linewidth]{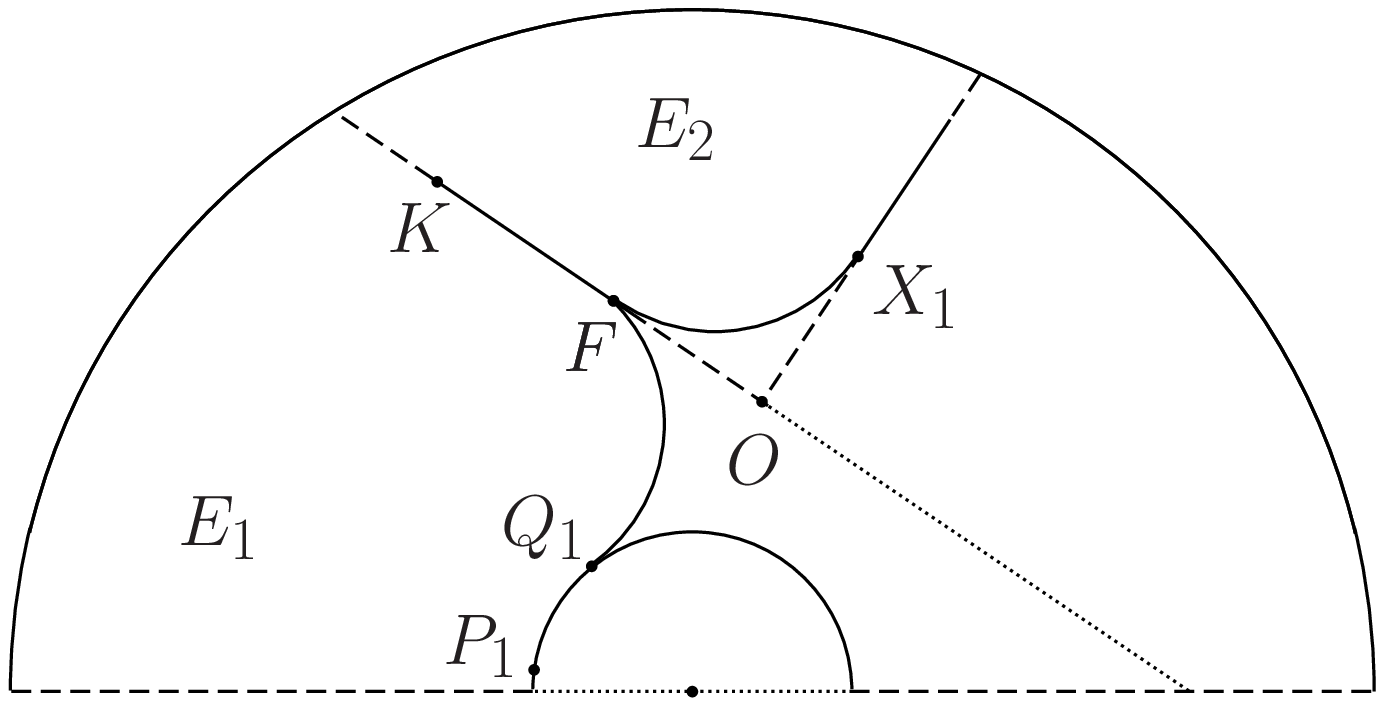}
		\caption{}
		\label{fig:Fig5}
	\end{minipage}
\end{figure}

\ref{proof:annuli:vi} Let us finish the proof by showing that the reference configuration is optimal. 
Since $c_{12} = 0$ by step \ref{proof:annuli:v}, we see from step \ref{proof:annuli:ii} that the arcs $P_1 Q_1 := \partial E_1 \cap \partial B_R$ and $P_2 Q_2 := \partial E_2 \cap \partial B_R$ are nonempty and they are the only parts of $\partial E_1$ and $\partial E_2$ with strictly negative curvature. 

Take, without loss of generality, an arc $P_1 Q_1$, and let $F \in \partial(E_1 E_2)$ be such that the arc $Q_1 F \subset \partial_f E_1$. 
We want to prove that the arc $FK \subset \partial (E_1 E_2)$ is ``perpendicular'' to $\partial B_R$. 
Let $X_1$ be the corresponding end point of the arc $F X_1 \subset \partial_f E_2$. 
Then there are three possibilities for a position of $X_1$:
\begin{enumerate}[label=(\alph*)]
	\item $X_1 \in \partial B_R$. In this case $FK$ is ``perpendicular'' to $\partial B_R$. 
	\item\label{enum:b} $X_1 \in \partial (E_1 E_2)$. 
	Due to piecewise nonnegativity of the curvature of $\partial E_2$ (except of the part $\partial E_2 \cap \partial B_R$), we see that $E_2$ must be contained inside the intersection of $\Omega$ with a cone based on the angle $\angle FOX_1$. 
	At the same time, this cone is contained inside a small segment of $B_1$ based on the line $FK$. 
	However, it is not hard to see that the reference configuration is better. A contradiction. 
	\item $X_1 \in \partial B_1$. 
	As in the case \ref{enum:b}, we see that $E_2$ must be contained in a small segment of $B_1$ based on the line $FK$. A contradiction. 
\end{enumerate}

Therefore, any segment of $\partial (E_1 E_2)$ which is connected with $\partial \Omega$ by an arc of a free boundary must be connected with both $\partial B_1$ and $\partial B_R$ and it is ``perpendicular'' to $\partial B_R$. 
Since, in view of step \ref{proof:annuli:ii}, $\partial E_i \cap \partial B_R$ has nonempty interior for $i=1,2$, we see that there are at least two segments of $\partial (E_1 E_2)$ which are connected with $\partial B_R$. Therefore, either $E_1$ or $E_2$ is contained inside a half-ring. From the uniqueness result of Lemma~\ref{lem:h1(Omega')} we conclude that $E_i$ coincides (up to rotation) with $\mathcal{O}_{r_0}$, and therefore the reference configuration is optimal.

\subsection{Second Cheeger constant for a disc}
Let $\Omega$ be a disc in $\mathbb{R}^2$ and let $(E_1, E_2) \in \mathcal{E}_2$ be a $2$-adjusted Cheeger couple of $\Omega$. 
Recall that $c_1$ and $c_1$ are the mean curvatures of the free boundaries $\partial_f E_1$ and $\partial_f E_2$, respectively, and $c_{12}$ is the mean curvature of the contact surface $\partial(E_1 E_2)$ measured from $E_1$.

It was shown in \cite{Parini} that each $E_i$ must be a first Cheeger set of a half-ball of $\Omega$. 
The last steps of the proof on \cite[pp.13-14]{Parini} rely on \cite[Theorem~3.9]{Parini}. However, as we described in Section \ref{sec:intro}, this theorem is not correct. 
Let us show that Proposition \ref{prop1} can be applied to overcome the usage of \cite[Theorem 3.9]{Parini}. 
From \cite[Section~4]{Parini} we know that $\partial(E_1E_2)$ has nonempty interior and cannot have closed connected components. Let $PQ$ be an arc of $\partial(E_1E_2)$ such that there are arcs $PE \in \partial_f E_1$ and $PF \in \partial_f E_2$, where $E,F \in \partial \Omega$, see \cite[Figure 3]{Parini}. 

If $c_1 = c_2$ and $c_{12} \neq 0$, then, as in \cite[p.~13]{Parini}, either $E_1$ or $E_2$ will be a subset of a circular segment strictly contained in a half-disc. However, the configuration with the Cheeger sets of the two half-disks is better, a contradiction. 
Therefore, either $c_1 = c_2$ and $c_{12} = 0$, or $c_1 \neq c_2$. 
Let us drop out the second case. 
Assume, without loss of generality, that $c_{12} \geq 0$. Then Proposition \ref{prop1} implies that $c_1 \geq c_2$. If we suppose that $c_1 > c_2$, then Proposition \ref{prop1} yields $c_1 = c_{12}$. This means that $E_1$ must be a ball of radius $r_1 = \frac{1}{c_1}$. However, it is impossible, since $h_1(E_1) = c_1$, but the Cheeger constant of the ball $E_1$ equals to $\frac{2}{r_1}$, that is, $h_1(E_1) = 2c_1$. Since $c_1 > 0$, we get a contradiction. 
Thus, the case where $c_1 = c_2$ and $c_{12} = 0$ is the only possible, and this directly leads to the optimal configuration.

\bigskip
\noindent
{\bf Acknowledgments.}
The article was started during a visit of E.P. at the University of West Bohemia and was continued during a visit of V.B. at Aix-Marseille University. The authors wish to thank the hosting institutions for the invitation and the kind hospitality. 
V.B. was supported by the project LO1506 of the Czech Ministry of Education, Youth and Sports.


\end{document}